\documentclass{article}
\usepackage{amssymb}

\usepackage{amsmath}

\newtheorem{theorem}{Theorem}

\newtheorem{corollary}[theorem]{Corollary}

\newtheorem{lemma}[theorem]{Lemma}

\newtheorem{proposition}[theorem]{Proposition}

\newenvironment{proof}[1][Proof]{\textbf{#1.} }{\ \rule{0.5em}{0.5em}}

\begin{document}

\title{Representations of $SO(3)$ and angular polyspectra}
\author{D. Marinucci and G. Peccati}
\maketitle

\begin{abstract}
We characterize the angular polyspectra, of arbitrary order, associated with
isotropic fields defined on the sphere $S^{2}=\left\{ \left( x,y,z\right)
:x^{2}+y^{2}+z^{2}=1\right\} $. Our techniques rely heavily on group
representation theory, and specifically on the properties of Wigner matrices
and Clebsch-Gordan coefficients. The findings of the present paper
constitute a basis upon which one can build formal procedures for the
statistical analysis and the probabilistic modelization of the Cosmic
Microwave Background radiation, which is currently a crucial topic of
investigation in cosmology. We also outline an application to random data
compression and \textquotedblleft simulation\textquotedblright\ of
Clebsch-Gordan coefficients.


\textbf{AMS\ 2000 Classification. }60G10; 60G35; 20C12; 20C35
\end{abstract}

\section{Introduction \label{intro}}

The connection between probability theory and group representation theory
has led to a long tradition of fruitful interactions. A well-known reference
is provided by \cite{Diaconis}; see e.g. \cite[Section 40-41]{bump}, \cite%
{DMWZZ}, \cite{Fulman}, \cite{GKR}, \cite{Pyc}, \cite{Raimond}, \cite{Yad},
and the references therein, for other relevant contributions. In this paper
we shall focus in particular on the profound connection between the
probabilistic notion of isotropy, i.e. invariance in law under the action of
a group, and the representation theory of the group itself. One instance of
this connection is well-known, i.e. the celebrated Peter-Weyl Theorem, which
allows the construction of spectral representations for isotropic random
fields on homogeneous spaces of general compact groups, see \cite{PePy} for
a general construction and \cite{MaPe} , \cite{MaPeSphere} for examples
related, respectively, to the torus and the sphere. Our aim here is to use
these representations in order to characterize random fields by means of a
higher order spectral theory; in particular, one of our main goals will be
to establish the link between the so-called \textsl{polyspectra} (or higher
order spectra) and alternative (tensor product and direct sum)
representations of the underlying isotropy group. In particular, we shall
provide a general expression for higher order spectra of isotropic spherical
random fields in terms of convolutions of Clebsch-Gordan or Wigner
coefficients. The latter where introduced in Mathematics in the XIX\ century
for the analysis of Algebraic Invariants; they have since then played a
crucial role in the development of Quantum Physics in the XX century (see
for instance \cite{VMK} for a comprehensive reference); their role in Group
Representation theory will be discussed below, while more details can be
found for instance in \cite{VilKly}.

Our analysis may have an intrinsic mathematical interest, but it is also
strongly motivated by applications to Physics and Cosmology. Concerning the
latter, the analysis of higher order spectra for isotropic spherical random
fields is currently at the core of several research efforts which are
related to the analysis of Cosmic Microwave Background (CMB) radiation data,
see for instance \cite{dodelson} for a general introduction and \cite%
{Hu,ks,m2006,MarPTRF} for some references on the bi- and trispectrum. A
general characterization of the theoretical properties of higher order
angular power spectra can yield several insights into the statistical
analysis of the massive datasets that are or will be made available by
satellite experiments such as \emph{WMAP} or \emph{Planck}. For instance,
the current understanding of the behaviour of the bispectrum for some simple
physical models has already led to many applications (\cite{Cabella}, \cite%
{yadav}, \cite{wandelt}), aiming at obtaining constraints on nonlinearity
parameters of utmost physical significance; needless to say, a proper
understanding of higher order spectra can lead to more efficient statistical
procedures and better constraints, which may help to solve some of the
important scientific issues at stake in CMB\ analysis (primarily a proper
understanding of the Big Bang \emph{inflationary} dynamics, which is tightly
linked with the CMB nonlinear structure, see \cite{dodelson}, \cite{creza}, %
\cite{Bart}, \cite{Maldac}).

The relevance of the current results need not be limited to cosmological
applications. Indeed, the analysis of spherical random fields has currently
led to remarkable developments in the Geophysical and Planetary Sciences,
and even in Medical Imaging, see for instance (\cite{chung}, \cite{simons}, %
\cite{wieczorek}). Moreover, we shall show below how the relationship which
we establish leads very naturally to some numerical algorithms for the
estimation of Clebsch-Gordan and Wigner coefficients. The latter represent
probability amplitudes of quantum interactions and as such a rich literature
in Mathematical Physics has been concerned with recipes for their numerical
estimation: our procedure lends itself to easy implementation and can be
simply extended to very general compact groups, although in this paper we
focus solely on $SO(3).$

The plan of this paper is as follows: in Section \ref{General} we introduce
our general probabilistic setting and provide some preliminary notation and
background material. In Section \ref{isotropy} we present some background
material on representation theory, while in Section \ref{S : POLY} and
Section \ref{S : APdelta} we obtain our main results, including the explicit
characterization of polyspectra. These results are applied in Section \ref%
{examples} to derive explicit expressions in some important cases (such as $%
\chi ^{2}$ random fields). Section \ref{further} is devoted to further
issues that we see as the seed for future research: they concern, in
particular, the connection with the representation theory for the symmetric
group and the Monte Carlo estimation of Clebsch-Gordan coefficients.

In the subsequent sections, every random element is defined on an
appropriate probability space $\left( \Omega ,\mathcal{F},P\right) $.

\section{General setting \label{General}}

In this paper, we focus on real-valued, centered, square-integrable and
isotropic random fields on the sphere $S^{2}=\left\{ \left( x,y,z\right) \in
\mathbb{R}^{3}:x^{2}+y^{2}+z^{2}=1\right\} $. A centered and square
integrable random field $T$ on $S^{2}$ is just a collection of random
variables of the type $T=\left\{ T(x):\text{ }x\in S^{2}\right\} $ such
that, for every $x\in S^{2}$, $ET(x)=0$ and $ET^{2}(x)<\infty $. In the
following, whenever we write that $T$ is a field on $S^{2}$, we will
implicitly assume that $T$ is real-valued, centered and square-integrable.
From now on, we shall distinguish between two notions of isotropy, which we
name \textsl{strong isotropy }and \textsl{weak isotropy of order }$n$ ($%
n\geq 2$).

\begin{description}
\item[Strong isotropy --] The field $T$ is said to be \textsl{strongly
isotropic} if, for every $k\in \mathbb{N}$, every $x_{1},...,x_{k}\in S^{2}$
and every $g\in SO(3)$ (the group of rotations in $\mathbb{R}^{3}$) we have%
\begin{equation}
\left\{ T(x_{1}),...T(x_{k})\right\} \overset{d}{=}\left\{
T(gx_{1}),...T(gx_{k})\right\} \text{ ,}  \label{idISO}
\end{equation}%
where $\overset{d}{=}$ denotes equality in distribution.

\item[Weak isotropy --] The field $T$ is said to be $n$-\textsl{weakly
isotropic} ($n\geq 2$) if $E|T(x)|^{n}<\infty $ for every $x\in S^{2}$, and
if, for every $x_{1},...,x_{n}\in S^{2}$ and every $g\in SO(3)$,
\begin{equation*}
E\left[ T(x_{1})\times \cdot \cdot \cdot \times T(x_{n})\right] =E\left[
T(gx_{1})\times \cdot \cdot \cdot \times T(gx_{n})\right] \text{.}
\end{equation*}
\end{description}

\bigskip

The following statement, whose proof is elementary, indicates some relations
between the two notions of isotropy described above.

\begin{proposition}
\label{P : EasyIso}

\begin{enumerate}
\item A strongly isotropic field with finite moments of some order $n\geq 2$
is also $n$-weakly isotropic.

\item Suppose that the field $T$ is $n$-weakly isotropic for every $n\geq 2$
(in particular, $E|T(x)|^{n}<\infty $ for every $n\geq 2$ and every $x\in
S^{2}$) and that, for every $k\geq 1$ and every $\left(
x_{1},...,x_{k}\right) $, the law of the vector $\left\{ T\left(
x_{1}\right) ,...,T\left( x_{k}\right) \right\} $ is determined by its
moments. Then, $T$ is also strongly isotropic.
\end{enumerate}
\end{proposition}

\bigskip

Now suppose that $T$ is a strongly isotropic field, and denote by $dx$ the
Lebesgue measure on $S^{2}$. Since the variance $ET\left( x\right) ^{2}$ is
finite and independent of $x$ (by isotropy), one deduces immediately that
\begin{equation*}
E\left[ \int_{S^{2}}T\left( x\right) ^{2}dx\right] <\infty \text{,}
\end{equation*}%
from which one infers that the random path $x\rightarrow T\left( x\right) $
is a.s. square integrable with respect to the Lebesgue measure. Then, it is
a standard result that the following spectral representation holds:%
\begin{equation}
T\left( x\right) =\sum_{l=0}^{\infty }\sum_{m=-l}^{l}a_{lm}Y_{lm}\left(
x\right) \text{, \ \ where \ \ }a_{lm}\triangleq \int_{S^{2}}T\left(
x\right) \overline{Y_{lm}\left( x\right) }dx\text{,}  \label{specrap}
\end{equation}%
and where the complex-valued functions $\left\{ Y_{lm}:l\geq 0\text{, \ }%
m=-l,...,l\right\} $ are the so-called \textsl{spherical harmonics}, to be
defined below. The spectral representation (\ref{specrap}) must be
understood in the $L^{2}(\Omega \times S^{2})$ sense, i.e.
\begin{equation*}
\lim_{L\rightarrow \infty }E\left\Vert
T-\sum_{l=0}^{L}\sum_{m=-l}^{l}a_{lm}Y_{lm}\right\Vert _{L^{2}(S^{2})}^{2}=0%
\text{,}
\end{equation*}%
where $L^{2}(S^{2})$ is the complex Hilbert space of functions on $S^{2}$,
which are square-integrable with respect to $dx$. If moreover the
trajectories of $T(x)$ are a.s. continuous, then the representation (\ref%
{specrap}) holds pointwise, i.e.
\begin{equation*}
\lim_{L\rightarrow \infty }\left\{
T(x)-\sum_{l=0}^{L}\sum_{m=-l}^{l}a_{lm}Y_{lm}\left( x\right) \right\} =0%
\text{ \ for all }x\in S^{2}\text{, \ a.s.-}P\text{,}
\end{equation*}%
see for instance \cite{adlertaylor}\textbf{\ }or\textbf{\ }\cite{Yad}. The
spherical harmonics $\left\{ Y_{lm}\right\} _{m=-l,...,l}$ are the
eigenfunctions of the Laplace-Beltrami operator on the sphere, denoted by $%
\Delta _{S^{2}}$, satisfying the relation $\Delta
_{S^{2}}Y_{lm}=-l(l+1)Y_{lm}.$ These functions can be represented by means
of spherical coordinates $x=(\theta ,\varphi )$ as follows:%
\begin{eqnarray*}
Y_{lm}(\theta ,\varphi ) &=&\sqrt{\frac{2l+1}{4\pi }\frac{(l-m)!}{(l+m)!}}%
P_{lm}(\cos \theta )\exp (im\varphi )\text{ , for }m>0\text{ ,} \\
Y_{lm}(\theta ,\varphi ) &=&(-1)^{m}\overline{Y_{l,-m}}(\theta ,\varphi )%
\text{ , for }m<0\text{ },\text{ }0\leq \theta \leq \pi ,\text{ }0\leq
\varphi <2\pi \text{ ,}
\end{eqnarray*}%
where $P_{lm}(\cos \theta )$ denotes the associated Legendre polynomial of
degree $l,m,$ i.e.
\begin{eqnarray*}
P_{lm}(x) &=&(-1)^{m}(1-x^{2})^{m/2}\frac{d^{m}}{dx^{m}}P_{l}(x)\text{ , }%
P_{l}(x)=\frac{1}{2^{l}l!}\frac{d^{l}}{dx^{l}}(x^{2}-1)^{l}, \\
m &=&0,1,2,...,l\text{ , }l=0,1,2,3,....\text{ .}
\end{eqnarray*}%
The random spherical harmonics coefficients $\left\{ a_{lm}\right\} $
appearing in (\ref{specrap}) form a triangular array of zero-mean and
square-integrable random variables, which are complex-valued for $m\neq 0$
and such that $Ea_{lm}\overline{a_{l^{\prime }m^{\prime }}}=\delta
_{l}^{l\prime }\delta _{m}^{m\prime }C_{l}$, the bar denoting complex
conjugation. Here, and for the rest of the paper, the symbol $\delta
_{b}^{a} $ is equal to one if $a=b$ and zero otherwise. We also write $%
C_{l}=E\left\vert a_{lm}\right\vert ^{2}$, $l\geq 0$, to indicate the
\textsl{angular power spectrum} of $T$ (we stress that the quantity $C_{l}$
does not depend on $m$ -- see e.g. \cite{BaMa} for a proof of this fact).
Observe that, by definition of the spherical harmonics, $a_{lm}=(-1)^{m}%
\overline{a_{l-m}}$. Note also that a convenient route to derive (\ref%
{specrap}) is by means of an appropriate version of the \textsl{stochastic
Peter-Weyl theorem} -- see for instance \cite{BaMaVa} or \cite{PePy}, as
well as Section \ref{SS : Represent SO(3)}\ below.

\bigskip

Observe that the representation (\ref{specrap}) still holds for fields $%
\left\{ T(x)\right\} $ that are not necessarily isotropic, but such that the
random path $x\rightarrow T(x)$ is $P$-a.s. square integrable with respect
to the Lebesgue measure $dx$. Indeed, if the last property holds, then one
has that, $P$-almost surely,
\begin{equation}
\lim_{L\rightarrow \infty }\int_{S^{2}}\left(
T(x)-\sum_{l=0}^{L}\sum_{m=-l}^{l}a_{lm}Y_{lm}\left( x\right) \right)
^{2}dx=0\text{.}  \label{pathrep}
\end{equation}%
In this case, however, none of the previously stated properties on the array
$\left\{ a_{lm}\right\} $ holds in general. By an argument similar to those
displayed above, a sufficient condition to have that $x\rightarrow T(x)$ is $%
P$-a.s. Lebesgue-square integrable is that $\sup_{x\in
S^{2}}ET(x)^{2}<\infty $.

The next result, that we record for future reference, is proved in \cite%
{BaMa}.

\begin{proposition}
\label{P : BaldiMar}Let $T$ be a centered, square-integrable and strongly
isotropic random field. Let the coefficients $\left\{ a_{lm}\right\} $ be
defined according to (\ref{specrap}). Then, for every $l,m$, one has that $%
E\left\vert a_{lm}\right\vert ^{2}<\infty $. Moreover, for every $l\geq 1$,
the coefficients $\left\{ a_{l0},...,a_{ll}\right\} $ are independent if,
and only if, they are Gaussian. If the vector $\left\{
a_{l0},...,a_{ll}\right\} $ is Gaussian, one also has that $\Re \left(
a_{lm}\right) $ and $\Im \left( a_{lm}\right) $ are independent and
identically distributed for every fixed $m=1,...,l$ ($\Re \left( z\right) $
and $\Im \left( z\right) $ stand, respectively, for the real and imaginary
parts of $z$).
\end{proposition}

\bigskip

The following result formalizes the fact that, in general, one cannot deduce
strong isotropy from weak isotropy. The proof makes use of Proposition \ref%
{P : EasyIso}.

\begin{proposition}
For every $n\geq 2$, there exists a $n$-weakly isotropic field $T$ such that
$T$ is not strongly isotropic.
\end{proposition}

\begin{proof}
Fix $l\geq 1$, and consider a vector%
\begin{equation*}
b_{m}\text{, \ \ }m=-l,...,l\text{,}
\end{equation*}%
of centered complex-valued random variables such that: (i) $b_{0}$ is real,
(ii) $b_{-m}=\left( -1\right) ^{m}\overline{b_{m}}$ ($m=1,...,l$), (iii) the
vector $\left\{ b_{0},...,b_{l}\right\} $ is not Gaussian and is composed of
independent random variables, (iv) for every $k=1,...,n$, the (possibly
mixed) moments of order $k$ of the variables $\left\{
b_{0},...,b_{l}\right\} $ coincide with those of a vector $\left\{
a_{0},...,a_{l}\right\} $ of independent, centered and complex-valued
Gaussian random variables with common variance $C_{l}$ and such that $a_{0}$
is real and, for every $m=1,...,l$, the real and imaginary parts of $a_{m}$
are independent and identically distributed (the existence of a vector such
as $\left\{ b_{0},...,b_{l}\right\} $ is easily proved). Now define the two
fields
\begin{equation*}
T\left( x\right) =\sum_{m=-l}^{l}b_{m}Y_{lm}\left( x\right) \text{ \ and \ }%
T^{\ast }\left( x\right) =\sum_{m=-l}^{l}a_{m}Y_{lm}\left( x\right) .
\end{equation*}%
By Proposition \ref{P : BaldiMar}, $T^{\ast }$ is strongly isotropic, and
also $n$-weakly isotropic by Proposition \ref{P : EasyIso}. By construction,
one also has that $T$ is $n$-weakly isotropic. However, $T$ cannot be
strongly isotropic, since this would violate Proposition \ref{P : BaldiMar}
(indeed, if $T$ was isotropic, one would have an example of an isotropic
field whose harmonic coefficients $\left\{ b_{0},...,b_{l}\right\} $ are
independent and non-Gaussian).
\end{proof}

\bigskip

In what follows, we use the symbol $A\otimes B$ to indicate the \textsl{%
Kronecker product} between two matrices $A$ and $B$. Given $n\geq 2$, we
denote by $\Pi \left( n\right) $ the class of partitions of the set $\left\{
1,...,n\right\} $. Given an element $\pi \in \Pi \left( n\right) $, we write
$\pi =\left\{ b_{1},...,b_{k}\right\} $ to indicate that the sets $%
b_{j}\subseteq \left\{ 1,...,n\right\} $, $j=1,...,k$, are the \textsl{blocks%
} of $\pi $. The blocks of a partition are always listed according to the
lexicographic order, that is: the block $b_{1}$ always contains $1$, the
block $b_{2}$ contains the least element of $\left\{ 1,...,n\right\} $ not
contained in $b_{1}$, and so on. Also the elements within each block $b_{j}$
are written in increasing order. For instance, if a partition $\pi $ of $%
\left\{ 1,...,5\right\} $ is composed of the blocks $\left\{ 1,3\right\}
,\left\{ 5,4\right\} $ and $\left\{ 2\right\} $, we will write $\pi $ in the
form $\pi =\left\{ \left\{ 1,3\right\} ,\left\{ 2\right\} ,\left\{
4,5\right\} \right\} .$

\bigskip

\noindent\textbf{Definition A. }(\textbf{A1}) Let the field $T$ admit the
representation (\ref{specrap}), and suppose that, for some $n\geq 2$, one
has that $E\left\vert a_{lm}\right\vert ^{n}<\infty $ for every $l,m$. Then,
$T$ is said to have \textsl{finite spectral moments }of order $n$.

(\textbf{A2}) Suppose that $T$ has finite spectral moments of order $n\geq 2$%
, and, for $l\geq 0$, use the notation%
\begin{equation}
a_{l.}=\left( a_{l-l},...,a_{l0},...,a_{ll}\right) .  \label{vecnot}
\end{equation}%
The \textsl{polyspectrum of order} $n-1$, associated with $T$, is given by
the collection of vectors
\begin{equation}
S_{l_{1}...l_{n}}=E\left[ a_{l_{1}.}\otimes a_{l_{2}.}\otimes \cdot \cdot
\cdot \otimes a_{l_{n}.}\right] \text{,}  \label{p}
\end{equation}%
where $0\leq l_{1},l_{2},...,l_{n}$. Note that the vector $S_{l_{1}...l_{n}}$
appearing in (\ref{p}) has dimension $\left( 2l_{1}+1\right) \times \cdot
\cdot \cdot \times \left( 2l_{n}+1\right) $.

(\textbf{A3}) Suppose that $T$ has finite spectral moments of order $n\geq 2$%
. The (mixed) \textsl{cumulant polyspectrum of order} $n-1$, associated with
$T$, is given by the vectors
\begin{equation}
S_{l_{1}...l_{n}}^{c}=\sum_{\pi =\left\{ b_{1},...,b_{k}\right\} \in \Pi
\left( n\right) }\left( -1\right) ^{k-1}\left( k-1\right) !E\left[ \otimes
_{i\in b_{1}}a_{l_{i}.}\right] \otimes \cdot \cdot \cdot \otimes E\left[
\otimes _{i\in b_{k}}a_{l_{i}.}\right] \text{,}  \label{pp}
\end{equation}%
where $0\leq l_{1},l_{2},...,l_{n}$, and, for every block $b_{j}=\left\{
i_{1},...,i_{p}\right\} $, we use the notation%
\begin{equation*}
E\left[ \otimes _{i\in b_{j}}a_{l_{i}.}\right] =E\left[ a_{l_{i_{1}.}}%
\otimes \cdot \cdot \cdot \otimes a_{l_{i_{p}.}}\right]
\end{equation*}%
(recall that we always list the elements of $b_{j}$ in such a way that $%
i_{1}\leq \cdot \cdot \cdot \leq i_{p}$). Plainly, the vector $%
S_{l_{1}...l_{n}}^{c}$ in (\ref{pp}) has also dimension $\left(
2l_{1}+1\right) \times \cdot \cdot \cdot \times \left( 2l_{n}+1\right) .$

\bigskip

\noindent \textbf{Remark. } Suppose that $T$ has finite spectral moments of
order $n\geq 2$. Then, by selecting frequencies $l_{1}=l_{2}=\cdot \cdot
\cdot =l_{3}=l\geq 0$, one obtains that%
\begin{equation}
S_{\underset{n\text{ times}}{\underbrace{l...l}}}^{c}:=S_{l...l}^{c}\left(
n\right) =\sum_{\pi =\left\{ b_{1},...,b_{k}\right\} \in \Pi \left( n\right)
}\left( -1\right) ^{k-1}\left( k-1\right) !E\left[ \left( a_{l.}\right)
^{\otimes \left\vert b_{1}\right\vert }\right] \otimes \cdot \cdot \cdot
\otimes E\left[ \left( a_{l.}\right) ^{\otimes \left\vert b_{k}\right\vert }%
\right]  \label{singleLcum}
\end{equation}%
where $\left\vert b_{j}\right\vert $ stands for the size of the block $b_{j}$%
, and we use the notation%
\begin{equation*}
\left( a_{l.}\right) ^{\otimes \left\vert b_{j}\right\vert }=\underset{%
\left\vert b_{j}\right\vert \text{ times}}{\underbrace{a_{l.}\otimes \cdot
\cdot \cdot \otimes a_{l.}}}\text{ \ .}
\end{equation*}

\section{Preliminary material \label{isotropy}}

\subsection{Representation Theory for $SO(3)$ \label{SS : Represent SO(3)}}

We start by reviewing briefly some background material on the special group
of rotations $SO(3)$, i.e. the space of $3\times 3$ real matrices $A$ such
that $A^{\prime }A=I_{3}$ (the three-dimensional identity matrix) and $\det
(A)=1$. We first\ recall that each element $g\in SO(3)$ can be parametrized
by the set $(\varphi ,\vartheta ,\psi )$ of the so-called \textsl{Euler
angles} ($0\leq \varphi <2\pi ,$ $0\leq \vartheta \leq \pi ,$ $0\leq \psi
<2\pi $); indeed each rotation in $\mathbb{R}^{3}$ can be realized
sequentially as%
\begin{equation}
A=A(g)=R(\psi ,\vartheta ,\varphi )=R_{z}(\psi )R_{x}(\vartheta
)R_{z}(\varphi )\text{ }  \label{ez}
\end{equation}%
where $R_{z}(\varphi ),R_{x}(\vartheta ),R_{z}(\psi )\in SO(3)$ can be
expressed by means of the following general definitions, valid for every
angle $\alpha $,%
\begin{equation*}
R_{z}(\alpha )=\left(
\begin{array}{ccc}
\cos \alpha & -\sin \alpha & 0 \\
\sin \alpha & \cos \alpha & 0 \\
0 & 0 & 1%
\end{array}%
\right) \text{ , }R_{x}(\alpha )=\left(
\begin{array}{ccc}
1 & 0 & 0 \\
0 & \cos \alpha & -\sin \alpha \\
0 & \sin \alpha & \cos \alpha%
\end{array}%
\right) \text{ .}
\end{equation*}%
The representation (\ref{ez}) is unique except for $\vartheta =0$ or $%
\vartheta =\pi $, in which case only the sum $\varphi +\psi $ is determined.
In words, the rotation is realized by rotating first by $\varphi $ around
the axis $z,$ then rotating around the new $x$ axis by $\vartheta ,$ then
rotating by $\psi $ around the new $z$ axis. It is clear that the first two
rotations identify one point on the sphere, so the whole operation could be
also interpreted as moving the North Pole to a new orientation in $S^{2}$
and then rotating by $\psi $ the tangent plane at the new location.

In these coordinates, a complete set of irreducible matrix representations
for $SO(3)$ is provided by the Wigner's $D$ matrices $D^{l}(\psi ,\vartheta
,\varphi )=\left\{ D_{mn}^{l}(\psi ,\vartheta ,\varphi )\right\}
_{m,n=-l,...,l}$, of dimensions $(2l+1)\times (2l+1)$ for $l=0,1,2,...;$ we
refer to classical textbooks, such as \cite{VilKly}, \cite{bump} or \cite%
{Diaconis}, for any unexplained definition or result concerning group
representation theory. An analytic expression for the elements of Wigner's $%
D $ matrices is provided by%
\begin{equation*}
D_{mn}^{l}(\psi ,\vartheta ,\varphi )=e^{-in\psi }d_{mn}^{l}(\vartheta
)e^{im\varphi },\text{ \ }m,n=-\left( 2l+1\right) ,...,2l+1
\end{equation*}%
where the indices $m,n$ indicate, respectively, columns and rows, and%
\begin{eqnarray*}
d_{mn}^{l}(\vartheta ) &=&(-1)^{l-n}\left[ (l+m)!(l-m)!(l+n)!(l-n)!\right]
^{1/2} \\
&&\times \sum_{k}(-1)^{k}\frac{\left( \cos \frac{\vartheta }{2}\right)
^{m+n+2k}\left( \sin \frac{\vartheta }{2}\right) ^{2l-m-n-2k}}{%
k!(l-m-k)!(l-n-k)!(m+n+k)!}\text{ ,}
\end{eqnarray*}%
and the sum runs over all $k$ such that the factorials are non-negative; see %
\cite[Chapter 4]{VMK} for a huge collection of alternative expressions. Here
we simply recall that the elements of $D^{l}(\psi ,\vartheta ,\varphi )$ are
related to the spherical harmonics by the relationship
\begin{equation}
D_{0m}^{l}(\varphi ,\vartheta ,\psi )=(-1)^{m}\sqrt{\frac{4\pi }{2l+1}}%
Y_{l-m}(\vartheta ,\varphi )=\sqrt{\frac{4\pi }{2l+1}}Y_{lm}^{\ast
}(\vartheta ,\varphi )\text{ .}  \label{spherwig}
\end{equation}%
In other words, the spherical harmonics correspond (up to a constant) to the
elements of the \textquotedblleft central\textquotedblright\ column in the
Wigner's $D$ matrix. Such matrices operate irreducibly and equivalently on $%
(2l+1)$ spaces (the so-called isotypical spaces), each of them spanned by a
different column $n$ of the matrix representation itself. The elements of
column $n$ correspond to the so-called \emph{spin n} spherical harmonics,
which enjoy a great importance in particle physics and in harmonic
expansions for tensor valued random fields. In this paper, we restrict our
attention only to the usual $n=0$\emph{\ }spherical harmonics, which
correspond to usual scalar functions.

\bigskip

\noindent \textbf{Remark. }By exploiting relation (\ref{spherwig}), it is
not difficult to show that the usual spectral representation for random
fields on the sphere, as given in (\ref{specrap}), is just the stochastic
Peter-Weyl Theorem on the quotient space $S^{2}=SO(3)/SO(2).$ Indeed, by the
stochastic Peter-Weyl Theorem (see e.g. \cite{PePy}) we obtain, for any
square integrable, isotropic random field $\left\{ T(g):\text{ }g\in
SO(3)\right\} $%
\begin{equation*}
T(g)=T(\varphi ,\vartheta ,\psi )=\sum_{l}\sum_{m,n}a_{lmn}\sqrt{\frac{2l+1}{%
8\pi ^{2}}}D_{mn}^{l}(\varphi ,\vartheta ,\psi )\text{ ,}
\end{equation*}%
where $dg$ is the Haar (uniform) measure on $SO\left( 3\right) $ with total
mass $8\pi ^{2}.$ Now if we consider the restriction of $T(g)$ to $%
S^{2}=SO(3)/SO(2)$, denoted by $T_{S^{2}}(\varphi ,\vartheta )$, we deduce
that%
\begin{eqnarray*}
a_{lmn} &=&\int_{SO(3)}T_{S^{2}}(g)\sqrt{\frac{2l+1}{8\pi ^{2}}}\overline{D}%
_{mn}^{l}(g)dg \\
&=&\int_{S^{2}}T_{S^{2}}(\varphi ,\vartheta )\left\{ \int_{0}^{2\pi
}e^{in\psi }d\psi \right\} \sqrt{\frac{2l+1}{8\pi ^{2}}}d_{mn}^{l}(\vartheta
)e^{-im\varphi }\sin \vartheta d\varphi d\vartheta \text{ ,} \\
&=&\int_{S^{2}}T_{S^{2}}(\varphi ,\vartheta )\delta _{n}^{0}(2\pi )\sqrt{%
\frac{2l+1}{8\pi ^{2}}}d_{mn}^{l}(\vartheta )e^{-im\varphi }\sin \vartheta
d\varphi d\vartheta \text{ ,}
\end{eqnarray*}%
the second equality following from the fact that $T_{S^{2}}(g)$ is constant
with respect to $\psi .$ We can thus conclude that%
\begin{equation*}
a_{lmn}=\left\{
\begin{array}{c}
0\text{ for }n\neq 0 \\
\sqrt{2\pi }a_{lm}\text{ for }n=0%
\end{array}%
\right. ,
\end{equation*}%
where the array $\left\{ a_{lm}\right\} $ is defined by (\ref{specrap}).

\subsection{The Clebsch-Gordan matrices \label{SS : CGmat}}

It follows from standard representation theory that we can exploit the
family $\left\{ D^{l}\right\} _{l=0,1,,2,...}$ to build alternative
(reducible) representations, either by taking the tensor product family $%
\left\{ D^{l_{1}}\otimes D^{l_{2}}\right\} _{l_{1},l_{2}}$, or by
considering direct sums $\left\{ \oplus
_{l=|l_{2}-l_{1}|}^{l_{2}+l_{1}}D^{l}\right\} _{l_{1},l_{2}}$. These
representations have dimensions
\begin{equation*}
(2l_{1}+1)(2l_{2}+1)\times (2l_{1}+1)(2l_{2}+1)
\end{equation*}
and are unitarily equivalent, whence there exists a unitary matrix $%
C_{l_{1}l_{2}}$ such that%
\begin{equation}
\left\{ D^{l_{1}}\otimes D^{l_{2}}\right\} =C_{l_{1}l_{2}}\left\{ \oplus
_{l=|l_{2}-l_{1}|}^{l_{2}+l_{1}}D^{l}\right\} C_{l_{1}l_{2}}^{\ast }\text{ .}
\label{clebun}
\end{equation}%
The matrix $C_{l_{1}l_{2}}$ is a $\left\{ (2l_{1}+1)(2l_{2}+1)\times
(2l_{1}+1)(2l_{2}+1)\right\} $ block matrix, whose blocks, of dimensions $%
(2l_{2}+1)\times (2l+1)$, are customarily denoted by $%
C_{l_{1}(m_{1})l_{2}}^{l}$, $m_{1}=-l_{1},...,l_{1};$ the elements of such a
block are indexed by $m_{2}$ (over rows) and $m$ (over columns; note that $%
m=-(2l+1),...,2l+1$). More precisely,%
\begin{eqnarray}
C_{l_{1}l_{2}} &=&\left[ C_{l_{1}(m_{1})l_{2}.}^{l.}\right]
_{m_{1}=-l_{1},...,l_{1};l=|l_{2}-l_{1}|,...,l_{2}+l_{1}}  \label{CG1} \\
C_{l_{1}(m_{1})l_{2}.}^{l.} &=&\left\{ C_{l_{1}m_{1}l_{2}m_{2}}^{lm}\right\}
_{m_{2}=-l_{2},...,l_{2};m=-l,...,l}\text{ .}  \label{CG2}
\end{eqnarray}

\bigskip

\noindent \textbf{Remark. }The fact that the two matrices $D^{l_{1}}\otimes
D^{l_{2}}$ and $\oplus _{l=|l_{2}-l_{1}|}^{l_{2}+l_{1}}D^{l}$ have the same
dimension follows from the elementary relation (valid for any integers $%
l_{1},l_{2}\geq 0$):%
\begin{equation}
\sum_{l=\left\vert l_{2}-l_{1}\right\vert }^{l_{1}+l_{2}}(2l+1)=\left(
2l_{1}+1\right) \left( 2l_{2}+1\right) \text{.}  \label{fd}
\end{equation}%
By induction, one also obtains that, for every $n\geq 3$,%
\begin{equation}
\sum_{\lambda _{1}=\left\vert l_{2}-l_{1}\right\vert
}^{l_{1}+l_{2}}\sum_{\lambda _{2}=\left\vert l_{3}-\lambda
_{1}\right\vert }^{\lambda _{1}+l_{3}}\cdot \cdot \cdot
\sum_{\lambda _{n-1}=\left\vert l_{n}-\lambda _{n-2}\right\vert
}^{\lambda _{n-2}+l_{n}}(2\lambda
_{n-1}+1)=\prod\limits_{j=1}^{n}\left( 2l_{j}+1\right) ,
\label{fdfd}
\end{equation}%
for any integers $l_{1},...,l_{n}\geq 0$ (relation (\ref{fdfd}) is needed in
Section \ref{SS : Structue}).

\bigskip

The \textsl{Clebsch-Gordan coefficients} for $SO(3)$ are then defined as the
collection $\left\{ C_{l_{1}m_{1}l_{2}m_{2}}^{lm}\right\} $ of the the
elements of the unitary matrices $C_{l_{1}l_{2}}$. These coefficients were
introduced in Mathematics in the XIX century, as motivated by the analysis
of invariants in Algebraic Geometry; in the 20th century, they have gained
an enormous importance in the quantum theory of angular momentum, where $%
C_{l_{1}m_{1}l_{2}m_{2}}^{lm}$ represents the \textsl{probability amplitude}
that two particles with total angular momentum $l_{1},l_{2}$ and momentum
projection on the $z$-axis $m_{1}$ and $m_{2}$ are coupled to form a system
with total angular momentum $l$ and projection $m$ (see e.g. \cite{Libo}).
Their use in the analysis of isotropic random fields is much more recent,
see for instance \cite{Hu} and the references therein.

\bigskip

\noindent \textbf{Remark }(\textit{More on the structure of the
Clebsch-Gordan matrices}). To ease the reading of the subsequent discussion,
we provide an alternative way of building a Clebsch-Gordan matrix $%
C_{l_{1}l_{2}}$, starting from any enumeration of its entries. Fix integers $%
l_{1},l_{2}\geq 0$ such that $l_{1}\leq l_{2}$ (this is just for notational
convenience), and consider the Clebsch-Gordan coefficients $\left\{
C_{l_{1}m_{1}l_{2}m_{2}}^{lm}\right\} $ given in (\ref{CG1})--(\ref{CG2}).
According to the above discussion, we know that: (i) $-l_{i}\leq m_{i}\leq
l_{i}$ for $i=1,2,$ (ii) $l_{2}-l_{1}$ $\leq l\leq l_{1}+l_{2}$, (iii) $%
-l\leq m\leq l$, and (iv) the symbols $\left( l_{1},m_{1},l_{2},m_{2}\right)
$ label rows, whereas the pairs $\left( l,m\right) $ are attached to
columns. Now introduce the total order $\prec _{c}$ on the \textquotedblleft
column pairs\textquotedblright\ $\left( l,m\right) $, by setting that $%
\left( l,m\right) \prec _{c}\left( l^{\prime },m^{\prime }\right) $,
whenever either $l<l^{\prime }$ or $l=l^{\prime }$ and $m<m^{\prime }$.
Analogously, introduce a total order $\prec _{r}$ over the \textquotedblleft
row symbols\textquotedblright\ $\left( l_{1},m_{1},l_{2},m_{2}\right) $, by
setting that $\left( l_{1},m_{1},l_{2},m_{2}\right) \prec _{r}\left(
l_{1}^{\prime },m_{1}^{\prime },l_{2}^{\prime },m_{2}^{\prime }\right) $, if
either $m_{1}<m_{1}^{\prime }$, or $m_{1}=m_{1}^{\prime }$ and $%
m_{2}<m_{2}^{\prime }$ (recall that $l_{1}$ and $l_{2}$ are fixed). One can
check that the set of column pairs (resp. row symbols) can now be written as
a \textsl{saturated chain}\footnote{%
Given a finite set $A=\left\{ a_{j}:j=1,...,N\right\} $ and an order $\prec $
on $A$, one says that $A$ is a \textsl{saturated chain} with respect to $%
\prec $ if there exists a permutation $\pi $ of $\left\{ 1,...,N\right\} $
such that
\begin{equation*}
a_{\pi \left( 1\right) }\prec a_{\pi \left( 2\right) }\prec \cdot \cdot
\cdot \prec a_{\pi \left( N-1\right) }\prec a_{\pi \left( N\right) }\text{.}
\end{equation*}%
In this case, $a_{\pi \left( 1\right) }$ and $a_{\pi \left( N\right) }$ are
called, respectively, the least and the maximal elements of the chain (see %
\cite[p. 99]{StaB})} with respect to $\prec _{c}$ (resp. $\prec _{r}$) with
a least element given by $\left( l_{2}-l_{1},-\left( l_{2}-l_{1}\right)
\right) $ (resp. $\left( l_{1},-l_{1},l_{2},-l_{2}\right) $) and a maximal
element given by $\left( l_{2}+l_{1},l_{2}+l_{1}\right) $ (resp. $\left(
l_{1},l_{1},l_{2},l_{2}\right) $). Then, (A) dispose the columns from west
to east, increasingly according to $\prec _{c}$, (B) dispose the rows from
north to south, increasingly according to $\prec _{r}$. For instance, by
setting $l_{1}=0$ and $l_{2}\geq 1$, one obtains that $C_{l_{1}l_{2}}$ is
the $\left( 2l_{2}+1\right) \times \left( 2l_{2}+1\right) $ square matrix $%
\left\{ C_{00l_{2}m_{2}}^{l_{2}m}\right\} $ with column indices $%
m=-(2l_{2}+1),...,(2l_{2}+1)$ and row indices $%
m_{2}=-(2l_{2}+1),...,(2l_{2}+1)$ (from the subsequent discussion, one also
deduces that, in general, $C_{00l_{2}m_{2}}^{lm}=\delta _{l}^{l_{2}}\delta
_{m}^{m_{2}}$). By selecting $l_{1}=l_{2}=1$, one sees that $C_{11}$ is the $%
9\times 9$ matrix with elements $C_{1m_{1}1m_{2}}^{lm}$ (for $%
m_{1},m_{2}=-1,0,1;$ $l=0,1,2,$ $m=-l,...,l)$ arranged as follows:%
\begin{equation*}
\left(
\begin{array}{ccccccccc}
\!C_{1,-1;1,-1}^{0,0}\! & \!C_{1,-1;1,-1}^{1,-1}\! & \!C_{1,-1;1,-1}^{10}\!
& \!C_{1,-1;1,-1}^{11}\! & \!C_{1,-1;1,-1}^{2,-2}\! & \!C_{1,-1;1,-1}^{2,-1}%
\! & \!C_{1,-1;1,-1}^{2,0}\! & \!C_{1,-1;1,-1}^{2,1}\! &
\!C_{1,-1;1,-1}^{2,2}\! \\
C_{1,-1;1,0}^{0,0} & ... & ... & ... & ... & ... & ... & ... & ... \\
C_{1,-1;1,1}^{0,0} & ... & ... & C_{1,-1;1,1}^{11} & ... & ... & ... & ... &
... \\
C_{1,0;1,-1}^{0,0} & ... & ... & ... & ... & ... & C_{1,0;1,-1}^{2,0} & ...
& ... \\
C_{1,0;1,0}^{0,0} & ... & ... & ... & C_{1,0;1,0}^{2,-2} & ... & ... & ... &
... \\
C_{1,0;1,1}^{0,0} & ... & ... & ... & ... & ... & ... & C_{1,0;1,1}^{2,1} &
... \\
C_{1,1;1,-1}^{0,0} & ... & ... & ... & ... & ... & ... & ... & ... \\
C_{1,1;1,0}^{0,0} & ... & ... & ... & ... & ... & ... & ... & ... \\
\!C_{1,1;1,1}^{0,0}\! & \!C_{1,1;1,1}^{1,-1}\! & \!C_{1,1;1,1}^{1,0}\! &
\!C_{1,1;1,1}^{1,1}\! & \!C_{1,1;1,1}^{2,-2}\! & \!C_{1,1;1,1}^{2,-1}\! &
\!C_{1,1;1,1}^{2,0}\! & \!C_{1,1;1,1}^{2,1}\! & \!C_{1,1;1,1}^{2,2}\!%
\end{array}%
\right)
\end{equation*}

\bigskip

\bigskip

Explicit expressions for the Clebsch-Gordan coefficients of $SO(3)$ are
known, but they are in general hardly manageable. We have for instance (see %
\cite{VMK}, expression 8.2.1.5)
\begin{align*}
C_{l_{1}m_{1}l_{2}m_{2}}^{l_{3}-m_{3}}& :=(-1)^{l_{1}+l_{3}+m_{2}}\sqrt{%
2l_{3}+1}\left[ \frac{%
(l_{1}+l_{2}-l_{3})!(l_{1}-l_{2}+l_{3})!(l_{1}-l_{2}+l_{3})!}{%
(l_{1}+l_{2}+l_{3}+1)!}\right] ^{1/2} \\
& \times \left[ \frac{(l_{3}+m_{3})!(l_{3}-m_{3})!}{%
(l_{1}+m_{1})!(l_{1}-m_{1})!(l_{2}+m_{2})!(l_{2}-m_{2})!}\right] ^{1/2} \\
& \times \sum_{z}\frac{(-1)^{z}(l_{2}+l_{3}+m_{1}-z)!(l_{1}-m_{1}+z)!}{%
z!(l_{2}+l_{3}-l_{1}-z)!(l_{3}+m_{3}-z)!(l_{1}-l_{2}-m_{3}+z)!}\text{ ,}
\end{align*}%
where the summation runs over all $z$'s such that the factorials are
non-negative. This expression becomes much neater for $m_{1}=m_{2}=m_{3}=0,$
where we have%
\begin{equation*}
C_{l_{1}0l_{2}0}^{l_{3}0}=\left\{
\begin{array}{c}
0\text{ , for }l_{1}+l_{2}+l_{3}\text{ odd} \\
(-1)^{\frac{l_{1}+l_{2}-l_{3}}{2}}\frac{\sqrt{2l_{3}+1}\left[
(l_{1}+l_{2}+l_{3})/2\right] !}{\left[ (l_{1}+l_{2}-l_{3})/2\right] !\left[
(l_{1}-l_{2}+l_{3})/2\right] !\left[ (-l_{1}+l_{2}+l_{3})/2\right] !}\left\{
\frac{(l_{1}+l_{2}-l_{3})!(l_{1}-l_{2}+l_{3})!(-l_{1}+l_{2}+l_{3})!}{%
(l_{1}+l_{2}+l_{3}+1)!}\right\} ^{1/2}\text{ , } \\
\text{for }l_{1}+l_{2}+l_{3}\text{ even}%
\end{array}%
\right. .
\end{equation*}%
The coefficients, moreover, enjoy a nice set of symmetry and orthogonality
properties, playing a crucial role in our results to follow. From unitary
equivalence we have the two relations:
\begin{eqnarray}
\sum_{m_{1},m_{2}}C_{l_{1}m_{1}l_{2}m_{2}}^{lm}C_{l_{1}m_{1}l_{2}m_{2}}^{l^{%
\prime }m^{\prime }} &=&\delta _{l}^{l\prime }\delta _{m}^{m\prime },
\label{ortho1} \\
\sum_{l,m}C_{l_{1}m_{1}l_{2}m_{2}}^{lm}C_{l_{1}m_{1}^{\prime
}l_{2}m_{2}^{\prime }}^{lm} &=&\delta _{m_{1}}^{m_{1}^{\prime }}\delta
_{m_{2}}^{m_{2}^{\prime }}\text{ ;}  \label{ortho2}
\end{eqnarray}%
in particular, (\ref{ortho1}) is a consequence of the orthogonality of row
vectors, whereas (\ref{ortho2}) comes from the orthogonality of columns.
Other properties are better expressed in terms of the Wigner's coefficients,
which are related to the Clebsch-Gordan coefficients by the identities (see %
\cite{VMK}, Chapter 8)%
\begin{eqnarray}
\left(
\begin{array}{ccc}
l_{1} & l_{2} & l_{3} \\
m_{1} & m_{2} & -m_{3}%
\end{array}%
\right) &=&(-1)^{l_{3}+m_{3}}\frac{1}{\sqrt{2l_{3}+1}}%
C_{l_{1}-m_{1}l_{2}-m_{2}}^{l_{3}m_{3}}  \label{clewig1} \\
C_{l_{1}m_{1}l_{2}m_{2}}^{l_{3}m_{3}} &=&(-1)^{l_{1}-l_{2}+m_{3}}\sqrt{%
2l_{3}+1}\left(
\begin{array}{ccc}
l_{1} & l_{2} & l_{3} \\
m_{1} & m_{2} & -m_{3}%
\end{array}%
\right) \text{ .}  \label{clewig2}
\end{eqnarray}

The Wigner's $3j$ (and, consequently, the Clebsch-Gordan) coefficients are
real-valued, they are different from zero only if $m_{1}+m_{2}+m_{3}=0$ and $%
l_{i}\leq l_{j}+l_{k}$ for all $i,j,k=1,2,3$ (\emph{triangle conditions}),
and they satisfy the symmetry conditions
\begin{equation*}
\left(
\begin{array}{ccc}
l_{1} & l_{2} & l_{3} \\
m_{1} & m_{2} & m_{3}%
\end{array}%
\right) =(-1)^{l_{1}+l_{2}+l_{3}}\left(
\begin{array}{ccc}
l_{1} & l_{2} & l_{3} \\
-m_{1} & -m_{2} & -m_{3}%
\end{array}%
\right) \text{ ,}
\end{equation*}

\begin{equation*}
\left(
\begin{array}{ccc}
l_{1} & l_{2} & l_{3} \\
m_{1} & m_{2} & m_{3}%
\end{array}%
\right) =(-1)^{sign(\pi )}\left(
\begin{tabular}{lll}
$l_{\pi (1)}$ & $l_{\pi (2)}$ & $l_{\pi (3)}$ \\
$m_{2}$ & $m_{3}$ & $m_{1}$%
\end{tabular}%
\right) \text{ ,}
\end{equation*}%
where $\pi $ is a permutation of $\left\{ 1,2,3\right\} $, and $sign\left(
\pi \right) $ denotes the sign of $\pi $. It follows also that for $%
m_{1}=m_{2}=m_{3}=0,$ the coefficients $C_{l_{1}0l_{2}0}^{l_{3}0}$ are
different from zero only when the sum $l_{1}+l_{2}+l_{3}$ is even. Later in
the paper, we shall also need the so-called Wigner's 6j coefficients, which
are defined by
\begin{equation}
\left\{
\begin{array}{ccc}
a & b & e \\
c & d & f%
\end{array}%
\right\} :=\sum_{\substack{ \alpha ,\beta ,\gamma  \\ \varepsilon ,\delta
,\phi }}(-1)^{e+f+\varepsilon +\phi }\left(
\begin{array}{ccc}
a & b & e \\
\alpha & \beta & \varepsilon%
\end{array}%
\right) \left(
\begin{array}{ccc}
c & d & e \\
\gamma & \delta & -\varepsilon%
\end{array}%
\right) \left(
\begin{array}{ccc}
a & d & f \\
\alpha & \delta & -\phi%
\end{array}%
\right) \left(
\begin{array}{ccc}
c & b & f \\
\gamma & \beta & \phi%
\end{array}%
\right) \text{ },  \label{6j1}
\end{equation}%
see \cite{VMK}, chapter 9 for analytic expressions and a full set of
properties; we simply recall here that the Wigner's 6j coefficients can
themselves be given an important interpretation in terms of group
representations, namely they relate different \emph{coupling schemes }in the
decomposition of tensor product into direct sum representations, see \cite%
{BieLou} for further details.

\bigskip

For future reference, we also recall some further standard properties of
Kronecker (tensor) products and direct sums of matrices: we have
\begin{eqnarray}
\oplus _{i=1}^{n}\left( A_{i}B_{i}\right) &=&\left( \oplus
_{i=1}^{n}A_{i}\right) \left( \oplus _{i=1}^{n}B_{i}\right) \text{ ,}
\label{KronSum} \\
\left( \oplus _{i=1}^{n}A_{i}\right) \otimes B &=&\oplus _{i=1}^{n}\left(
A_{i}\otimes B\right)  \label{KronDistr}
\end{eqnarray}%
and, provided all matrix products are well-defined,%
\begin{equation}
\left( AB\otimes C\right) =\left( A\otimes I_{n}\right) \left( B\otimes
C\right) \text{ .}  \label{KronDec}
\end{equation}%
Here, $\oplus _{i=1}^{n}A_{i}$ is defined as the block diagonal matrix $%
diag\left\{ A_{1},...,A_{n}\right\} $ if $A_{i}$ is a set of square matrices
of order $r_{i}\times r_{i}$, whereas it is defined as the stacked column
vector of order $\left( \sum_{i=1}^{n}r_{i}\right) \times 1$ if the $A_{i}$
are $r_{i}\times 1$ column vectors.

\section{Characterization of polyspectra\label{S : POLY}}

\subsection{Four general statements}

The following result is well-known. As it is crucial in our arguments to
follow and we failed to locate any explicit reference, we shall provide a
short proof for the sake of completeness. Note that, in the sequel, we use
the symbol $a_{l.}$ to indicate the $\left( 2l+1\right) $-dimensional
complex-valued random vector defined in (\ref{vecnot})$.$

\begin{lemma}
\label{L : RandomWigner}Let $T$ be a strongly isotropic field on $S^{2}$,
and let the harmonic coefficients $\left\{ a_{lm}\right\} $ be defined
according to (\ref{specrap}). Then, for every $l\geq 0$ and every $g\in
SO(3) $, we have%
\begin{equation}
D^{l}(g)a_{l.}\overset{d}{=}a_{l.}\text{ , }l=0,1,2,...\text{ .}  \label{uu}
\end{equation}%
The equality (\ref{uu}) must be understood in the sense of
finite-dimensional distributions for sequences of random vectors, that is, (%
\ref{uu}) takes place if, and only if, for every $k\geq 1$ and every $0\leq
l_{1}<l_{2}<\cdot \cdot \cdot <l_{k}$,%
\begin{equation}
\left\{ D^{l_{1}}(g)a_{l_{1}.},...,D^{l_{k}}(g)a_{l_{k}.}\right\} \overset{d}%
{=}\left\{ a_{l_{1}.},...,a_{l_{k}.}\right\} .  \label{uuu}
\end{equation}
\end{lemma}

\begin{proof}
We provide the proof of (\ref{uuu}) only when $k=1$ and $l_{1}=l\geq 1$. The
general case is obtained analogously. By strong isotropy, we have that, for
every $l\geq 1$, every $g\in SO\left( 3\right) $ and every $%
x_{1},...,x_{n}\in S^{2}$, the equality (\ref{idISO}) takes place. Now, (\ref%
{idISO}) can be rewritten as follows:%
\begin{eqnarray}
&&\left\{
\sum_{l}\sum_{m}a_{lm}Y_{lm}(x_{1}),...,\sum_{l}\sum_{m}a_{lm}Y_{lm}(x_{n})%
\right\} \overset{d}{=}\left\{
\sum_{l}\sum_{m}a_{lm}Y_{lm}(gx_{1}),...,\sum_{l}%
\sum_{m}a_{lm}Y_{lm}(gx_{n})\right\}  \notag \\
&=&\left\{ \sum_{l}\sum_{m}a_{lm}\sum_{m^{\prime }}D_{m^{\prime
}m}^{l}(g)Y_{lm^{\prime }}(x_{1}),...,\sum_{l}\sum_{m}a_{lm}\sum_{m^{\prime
}}D_{m^{\prime }m}^{l}(g)Y_{lm^{\prime }}(x_{n})\right\}  \notag \\
&=&\left\{ \sum_{l}\sum_{m^{\prime }}\widetilde{a}_{lm^{\prime
}}Y_{lm^{\prime }}(x_{1}),...,\sum_{l}\sum_{m^{\prime }}\widetilde{a}%
_{lm^{\prime }}Y_{lm^{\prime }}(x_{n})\right\} \text{ ,}  \label{j}
\end{eqnarray}%
where we write%
\begin{equation}
\widetilde{a}_{lm^{\prime }}\triangleq \sum_{m}a_{lm}D_{m^{\prime }m}^{l}(g)%
\text{,}  \label{jj}
\end{equation}%
and we have used
\begin{equation}
\left\{ Y_{lm}(gx_{1}),...,Y_{lm}(gx_{n})\right\} \equiv \left\{
\sum_{m^{\prime }}D_{m^{\prime }m}^{l}(g)Y_{lm^{\prime
}}(x_{1}),...,\sum_{m^{\prime }}D_{m^{\prime }m}^{l}(g)Y_{lm^{\prime
}}(x_{n})\right\} .  \label{jjj}
\end{equation}%
which follows from the group representation property and the identity (\ref%
{spherwig}). To conclude, just observe that (\ref{j}) implies that%
\begin{equation*}
\widetilde{a}_{lm^{\prime }}=\int_{S^{2}}T\left( gx\right) \overline{%
Y_{lm^{\prime }}\left( x\right) }dx\text{, \ \ }m^{\prime }=-l,...,l\text{,}
\end{equation*}%
yielding that, due to strong isotropy and\ with obvious notation, $%
\widetilde{a}_{l.}\overset{d}{=}a_{l.}$. The conclusion follows from the
fact that, thanks to (\ref{jj}),
\begin{equation*}
\widetilde{a}_{l.}=D^{l}\left( g\right) a_{l.}\text{.}
\end{equation*}
\end{proof}

The next theorem connects the invariance properties of the vectors $\left\{
a_{l.}\right\} $ to the representations of $SO(3)$. We need first to
establish some notation. For every $0\leq l_{1},l_{2},...,l_{n}$, we shall
write
\begin{eqnarray}
\Delta _{l_{1}...l_{n}} &\triangleq &\int_{SO(3)}\left\{ D^{l_{1}}(g)\otimes
D^{l_{2}}(g)\otimes ...\otimes D^{l_{n}}(g)\right\} dg\text{ ,}  \label{not1}
\\
\Delta _{l_{1}...l_{n}}\left( g\right) &\triangleq &D^{l_{1}}(g)\otimes
D^{l_{2}}(g)\otimes ...\otimes D^{l_{n}}(g)\text{, \ \ \ }g\in SO\left(
3\right) \text{,}  \label{notN2}
\end{eqnarray}%
and use the symbol $S_{l_{1}...l_{n}}$ (whenever is well-defined), as given
in formula (\ref{p}). We stress that $\Delta _{l_{1}...l_{n}}$ and $\Delta
_{l_{1}...l_{n}}\left( g\right) $ are square matrices with $(2l_{1}+1)\times
...\times (2l_{n}+1)$ rows and $S_{l_{1}...l_{n}}$ is a column vector with $%
(2l_{1}+1)\times ...\times (2l_{n}+1)$ elements. \ The following result
applies to an arbitrary $n\geq 2$: see \cite{Hu} for some related results in
the case $n=3,4.$

\begin{proposition}
\label{propa} Let $T$ be a strongly isotropic field with moments of order $%
n\geq 2$. Then, for every $0\leq l_{1},l_{2},...,l_{n}$ and every fixed $%
g^{\ast }\in SO\left( 3\right) $%
\begin{eqnarray}
\Delta _{l_{1}...l_{n}}S_{l_{1}...l_{n}} &=&S_{l_{1}...l_{n}}\text{ }
\label{crurel} \\
\Delta _{l_{1}...l_{n}}\left( g^{\ast }\right) S_{l_{1}...l_{n}}
&=&S_{l_{1}...l_{n}}.  \label{crurel2}
\end{eqnarray}%
On the other hand, fix $n\geq 2$ and assume that $T(x)$ is a not necessarily
isotropic random field on the sphere s.t. $\sup_{x}\left( E\left\vert
T(x)\right\vert ^{n}\right) <\infty $. Then $T\left( .\right) $ is $P$%
-almost surely Lebesgue square integrable and the $n$th order spectral
moments of $T$ exist and are finite. If moreover (\ref{crurel}) holds for
every $0\leq l_{1}\leq \cdot \cdot \cdot \leq l_{n}$, then one has that, for
every $g\in SO\left( 3\right) $,%
\begin{equation}
E\left[ D^{l_{1}}(g)a_{l_{1}.}\otimes \cdot \cdot \cdot \otimes
D^{l_{n}}(g)a_{l_{n}.}\right] =E\left[ a_{l_{1}.}\otimes \cdot \cdot \cdot
\otimes a_{l_{n}.}\right] \text{,}  \label{uuww}
\end{equation}%
and $T$ is $n$-weakly isotropic.
\end{proposition}

\begin{proof}
By strong\ isotropy and Lemma \ref{L : RandomWigner}, one has%
\begin{equation*}
E\left\{ D^{l_{1}}(g)a_{l_{1}.}\otimes ...\otimes
D^{l_{n}}(g)a_{l_{n}.}\right\} =E\left\{ a_{l_{1}.}\otimes ...\otimes
a_{l_{n}.}\right\} \text{ for all }g\in SO(3)\text{ , }l_{1},...,l_{n}\in
\mathbb{N}^{n}.
\end{equation*}%
Now assume that $g$ is sampled randomly (and independently of the $\left\{
a_{l.}\right\} $) according to some probability measure, say $P_{0}$, on $%
SO(3).$ From the property (\ref{KronDec}) of tensor products and trivial
manipulations, we obtain (with obvious notation and by independence)%
\begin{eqnarray*}
E\left\{ D^{l_{1}}(\cdot )a_{l_{1}.}\otimes ...\otimes D^{l_{n}}(\cdot
)a_{l_{n}.}\right\} &=&E\left\{ \left[ D^{l_{1}}(\cdot )\otimes ...\otimes
D^{l_{n}}(\cdot )\right] \left[ a_{l_{1}.}\otimes ...\otimes a_{l_{n}.}%
\right] \right\} \\
&=&E_{0}\left\{ D^{l_{1}}(\cdot )\otimes ...\otimes D^{l_{n}}(\cdot
)\right\} E\left\{ a_{l_{1}.}\otimes ...\otimes a_{l_{n}.}\right\} .
\end{eqnarray*}%
Now, if one chooses $P_{0}$ to be equal to the Haar (uniform) measure on $%
SO\left( 3\right) $, one has that
\begin{equation*}
E_{0}\left\{ D^{l_{1}}(\cdot )\otimes ...\otimes D^{l_{n}}(\cdot )\right\}
=\Delta _{l_{1}...l_{n}}\text{,}
\end{equation*}%
thus giving (\ref{crurel}). On the other hand, if one chooses $P_{0}$ to be
equal to the Dirac mass at some $g^{\ast }\in SO\left( 3\right) $, one has
that%
\begin{equation*}
E_{0}\left\{ D^{l_{1}}(\cdot )\otimes ...\otimes D^{l_{n}}(\cdot )\right\}
=\Delta _{l_{1}...l_{n}}\left( g^{\ast }\right) \text{,}
\end{equation*}%
which shows that (\ref{crurel2}) is satisfied.

\noindent Now let $T$ satisfy the assumptions of the second part of the
statement for some $n\geq 2$. We recall first that the representation (\ref%
{specrap}) continues to hold, in a pathwise sense. To see that the $n$th
order joint moments of the harmonic coefficients $a_{lm}$ are finite it is
enough to use Jensen's inequality, along with a standard version of the
Fubini theorem, to obtain that%
\begin{eqnarray*}
E\left\vert a_{lm}\right\vert ^{n} &=&E\left\vert \int_{S^{2}}T(x)\overline{%
Y_{lm}(x)}dx\right\vert ^{n}\leq E\int_{S^{2}}|T(x)|^{n}|Y_{lm}(x)|^{n}dx \\
&\leq &\left\{ \sup_{x\in S^{2}}|Y_{lm}(x)|^{n}\right\} \left\{ \sup_{x\in
S^{2}}E|T(x)|^{n}\right\} \\
&\leq &\left( \frac{2l+1}{4\pi }\right) ^{n/2}\left\{ \sup_{x\in
S^{2}}E|T(x)|^{n}\right\} <\infty \text{ .}
\end{eqnarray*}%
It is then straightforward that, if $S_{l_{1}...l_{n}}$ satisfies (\ref%
{crurel}), one also has that for any fixed $\overline{g}\in SO(3)$%
\begin{eqnarray*}
&&E\left\{ \left[ D^{l_{1}}(\overline{g})\otimes ...\otimes D^{l_{n}}(%
\overline{g})\right] \left[ a_{l_{1}.}\otimes ...\otimes a_{l_{n}.}\right]
\right\} \\
&=&\left[ D^{l_{1}}(\overline{g})\otimes ...\otimes D^{l_{n}}(\overline{g})%
\right] E\left[ a_{l_{1}.}\otimes ...\otimes a_{l_{n}.}\right] \\
&=&\left[ D^{l_{1}}(\overline{g})\otimes ...\otimes D^{l_{n}}(\overline{g})%
\right] \Delta _{l_{1}...l_{n}}S_{l_{1}...l_{n}} \\
&=&\left\{ \left[ D^{l_{1}}(\overline{g})\otimes ...\otimes D^{l_{n}}(%
\overline{g})\right] \int_{SO(3)}\left\{ D^{l_{1}}(g)\otimes ...\otimes
D^{l_{n}}(g)\right\} dg\right\} S_{l_{1}...l_{n}} \\
&=&\left\{ \int_{SO(3)}\left\{ D^{l_{1}}(\overline{g}g)\otimes D^{l_{2}}(%
\overline{g}g)\otimes ...\otimes D^{l_{n}}(\overline{g}g)\right\} dg\right\}
S_{l_{1}...l_{n}} \\
&=&\Delta _{l_{1}...l_{n}}S_{l_{1}...l_{n}}=E\left\{ a_{l_{1}.}\otimes
...\otimes a_{l_{n}.}\right\} \text{ ,}
\end{eqnarray*}%
which proves the $n$-th spectral moment is invariant to rotations. The fact
that $T$ is $n$-weakly isotropic is a consequence of the spectral
representation (\ref{specrap}).
\end{proof}

\bigskip

Note that relation (\ref{crurel}) can be rephrased by saying that, for a
strongly isotropic field, the joint moment vector $E\left\{
a_{l_{1}.}\otimes a_{l_{2}.}\otimes ...\otimes a_{l_{n}.}\right\} $ must be
an eigenvector of the matrix (\ref{not1}) for every $n\geq 2$ and every $%
0\leq l_{1}\leq \cdot \cdot \cdot \leq l_{n}$. A similar characterization
holds for cumulants polyspectra. Recall the notation $S_{l_{1}...l_{n}}^{c}$
introduced in (\ref{pp}).

\begin{proposition}
\label{P : Cum_as_eig} Let $T$ be a strongly isotropic field with moments of
order $n\geq 2$. Then, for every $0\leq l_{1},l_{2},...,l_{n}$ and every
fixed $g^{\ast }\in SO\left( 3\right) $,
\begin{eqnarray}
\Delta _{l_{1}...l_{n}}S_{l_{1}...l_{n}}^{c} &=&S_{l_{1}...l_{n}}^{c}\text{ }
\label{Spectre3} \\
\Delta _{l_{1}...l_{n}}\left( g^{\ast }\right) S_{l_{1}...l_{n}}^{c}
&=&S_{l_{1}...l_{n}}^{c}.  \label{Spectre 4}
\end{eqnarray}%
On the other hand, fix $n\geq 2$ and assume that $T(x)$ is a not necessarily
isotropic random field on the sphere s.t. $\sup_{x}\left( E\left\vert
T(x)\right\vert ^{n}\right) <\infty $. Then $T\left( .\right) $ is $P$%
-almost surely Lebesgue square integrable and the $n$th order spectral
moments of $T$ exist and are finite. If moreover (\ref{Spectre3}) holds for
every $0\leq l_{1}\leq \cdot \cdot \cdot \leq l_{n}$, then one has that, for
every $g\in SO\left( 3\right) $, relation (\ref{uuww}) holds, and $T$ is $n$%
-weakly isotropic.
\end{proposition}

\begin{proof}
For every $x_{1},...,x_{n}\in S^{2}$, write $\mathrm{Cum}\left\{ T\left(
x_{1}\right) ,...,T\left( x_{n}\right) \right\} $ the joint cumulant of the
random variables $T\left( x_{1}\right) ,...,T\left( x_{n}\right) $. By using
isotropy, one has that, for every $g\in SO\left( 3\right) $,%
\begin{equation}
\mathrm{Cum}\left\{ T\left( x_{1}\right) ,...,T\left( x_{n}\right) \right\} =%
\mathrm{Cum}\left\{ T\left( gx_{1}\right) ,...,T\left( gx_{n}\right)
\right\} .  \label{coo}
\end{equation}%
Hence, by using the well-known multilinearity properties of cumulants, one
deduces that (with obvious notation)%
\begin{eqnarray}
&&\mathrm{Cum}\left\{ T\left( x_{1}\right) ,...,T\left( x_{n}\right) \right\}
\notag \\
&=&\sum_{l_{1}m_{1},...,l_{n}m_{n}}\mathrm{Cum}\left\{
a_{l_{1}m_{1}},...,a_{l_{n}m_{n}}\right\} Y_{l_{1}m_{1}}\left( x_{1}\right)
\cdot \cdot \cdot Y_{l_{n}m_{n}}\left( x_{n}\right)  \notag \\
&=&\sum_{l_{1}m_{1},...,l_{n}m_{n}}\mathrm{Cum}\left\{
a_{l_{1}m_{1}},...,a_{l_{n}m_{n}}\right\} Y_{l_{1}m_{1}}\left( gx_{1}\right)
\cdot \cdot \cdot Y_{l_{n}m_{n}}\left( gx_{n}\right) ,
\label{penitenziagite}
\end{eqnarray}%
and relations (\ref{Spectre3})--(\ref{Spectre 4}) are deduced by rewriting (%
\ref{penitenziagite}) by means of the identity%
\begin{equation*}
\left\{ Y_{l_{1}m_{1}}(gx_{1}),...,Y_{l_{n}m_{n}}(gx_{n})\right\}
\equiv \left\{ \sum_{m^{\prime }}D_{m^{\prime
}m_1}^{l_{1}}(g)Y_{l_{1}m^{\prime }}(x_{1}),...,\sum_{m^{\prime
}}D_{m^{\prime }m_n}^{l_{n}}(g)Y_{l_{n}m^{\prime }}(x_{n})\right\} .
\end{equation*}

The second part of the statement is proved by arguments analogous to the
ones used in the proof of Proposition \ref{propa}.
\end{proof}

\bigskip

We now present an alternative (and more involved) characterization of the
cumulant polyspectra associated with an isotropic field. Given $n\geq 2$ and
a partition $\pi =\left\{ b_{1},...,b_{k}\right\} \in \Pi \left( n\right) $,
we build a permutation $v^{\pi }=\left( v^{\pi }\left( 1\right) ,...,v^{\pi
}\left( n\right) \right) \in \mathfrak{S}_{n}$ as follows: (i) write the
partition
\begin{equation}
\pi =\left\{ b_{1},...,b_{k}\right\} =\left\{ \left(
i_{1}^{1},...,i_{|b_{1}|}^{1}\right) ,...,\left(
i_{1}^{k},...,i_{|b_{k}|}^{k}\right) \right\}  \label{part}
\end{equation}%
(where $|b_{j}|$ $\geq 1$ stands for the size of $b_{j}$) by means of the
convention outlined in Section \ref{General} (that is, order the blocks and
the elements within each block according to the lexicographic order); (ii)
define $v^{\pi }=\mathfrak{S}_{n}$ by simply removing the brackets in (\ref%
{part}), that is, set
\begin{equation*}
v^{\pi }=\left( v^{\pi }\left( 1\right) ,...,v^{\pi }\left( n\right) \right)
=\left(
i_{1}^{1},...,i_{|b_{1}|}^{1},i_{1}^{2},...,i_{|b_{2}|}^{2},...,i_{1}^{k},...,i_{|b_{k}|}^{k}\right) .
\end{equation*}%
For instance, if a partition $\pi $ of $\left\{ 1,...,6\right\} $ is
composed of the blocks $\left\{ 1,3\right\} ,\left\{ 6,4\right\} $ and $%
\left\{ 2,5\right\} $, one first writes $\pi $ in the form $\pi =\left\{
\left\{ 1,3\right\} ,\left\{ 2,5\right\} ,\left\{ 4,6\right\} \right\} $,
and then defines $v^{\pi }=\left( v^{\pi }\left( 1\right) ,...,v^{\pi
}\left( 6\right) \right) =\left( 1,3,2,5,4,6\right) $. Given $n\geq 2$, $%
0\leq l_{1}\leq \cdot \cdot \cdot \leq l_{n}$, and $\pi \in \Pi \left(
n\right) $, we define the matrix
\begin{equation}
\Delta _{l_{1}...l_{n}}^{\pi }\triangleq \int_{SO(3)}\left\{ D^{l_{v^{\pi
}\left( 1\right) }}(g)\otimes D^{l_{v^{\pi }\left( 2\right) }}(g)\otimes
...\otimes D^{l_{v^{\pi }\left( n\right) }}(g)\right\} dg\text{ ,}
\label{notnot}
\end{equation}%
obtained from the matrix $\Delta _{l_{1}...l_{n}}$ in (\ref{not1}), by
permuting the indexes $l_{i}$ according to $v^{\pi }$. Plainly, if $v^{\pi }$
is equal to the identity permutation, then $\Delta _{l_{1}...l_{n}}^{\pi
}=\Delta _{l_{1}...l_{n}}$. We also set, for every fixed $g\in SO\left(
3\right) $,%
\begin{equation*}
\Delta _{l_{1}...l_{n}}^{\pi }\left( g\right) \triangleq D^{l_{v^{\pi
}\left( 1\right) }}(g)\otimes D^{l_{v^{\pi }\left( 2\right) }}(g)\otimes
...\otimes D^{l_{v^{\pi }\left( n\right) }}(g).
\end{equation*}

\begin{proposition}
\label{propa2} Let $T$ be a strongly isotropic field with finite moments of
order $n\geq 2$. For $0\leq l_{1},l_{2},...,l_{n}$, define $%
S_{l_{1}...l_{n}}^{c}$ according to (\ref{pp}). Then, for every $0\leq
l_{1},l_{2},...,l_{n}$, and every $g\in SO\left( 3\right) $%
\begin{eqnarray}
S_{l_{1},...l_{n}}^{c} &=&\sum_{\pi =\left\{ b_{1},...,b_{k}\right\} \in \Pi
\left( n\right) }\left( -1\right) ^{k-1}\left( k-1\right) !\Delta
_{l_{1}...l_{n}}^{\pi }E\left[ \otimes _{i\in b_{1}}a_{l_{i}.}\right]
\otimes \cdot \cdot \cdot \otimes E\left[ \otimes _{i\in b_{k}}a_{l_{i}.}%
\right] \text{ }  \label{johnChin} \\
&=&\sum_{\pi =\left\{ b_{1},...,b_{k}\right\} \in \Pi \left( n\right)
}\left( -1\right) ^{k-1}\left( k-1\right) !\Delta _{l_{1}...l_{n}}^{\pi
}\left( g\right) E\left[ \otimes _{i\in b_{1}}a_{l_{i}.}\right] \otimes
\cdot \cdot \cdot \otimes E\left[ \otimes _{i\in b_{k}}a_{l_{i}.}\right]
\label{JC2}
\end{eqnarray}%
On the other hand, fix $n\geq 2$ and assume that $T(x)$ is a (not
necessarily isotropic) random field on the sphere s.t. $\sup_{x}\left(
E\left\vert T(x)\right\vert ^{n}\right) <\infty $. Then, the $n$th order
spectral moments and cumulants of $T$ exist and are finite. If moreover (\ref%
{JC2}) holds for every $0\leq l_{1},l_{2},...,l_{n}$ and every $g\in
SO\left( 3\right) $, then one has that $T$ is $n$-weakly isotropic.
\end{proposition}

\begin{proof}
Fix $\pi =\left\{ b_{1},...,b_{k}\right\} \in \Pi \left( n\right) $. By
strong isotropy and Lemma \ref{L : RandomWigner}, one has that, for a fixed $%
g^{\ast }\in SO\left( 3\right) $, the quantity%
\begin{eqnarray*}
&&E\left[ \otimes _{i\in b_{1}}D^{l_{i}}\left( g\right) a_{l_{i}.}\right]
\otimes \cdot \cdot \cdot \otimes E\left[ \otimes _{i\in
b_{k}}D^{l_{i}}\left( g\right) a_{l_{i}.}\right] \\
&=&\Delta _{l_{1}...l_{n}}^{\pi }\left( g^{\ast }\right) E\left[ \otimes
_{i\in b_{1}}a_{l_{i}.}\right] \otimes \cdot \cdot \cdot \otimes E\left[
\otimes _{i\in b_{k}}a_{l_{i}.}\right]
\end{eqnarray*}%
does not depend on $g^{\ast }$, so that
\begin{eqnarray*}
&&E\left[ \otimes _{i\in b_{1}}a_{l_{i}.}\right] \otimes \cdot \cdot \cdot
\otimes E\left[ \otimes _{i\in b_{k}}a_{l_{i}.}\right] \\
&=&\Delta _{l_{1}...l_{n}}^{\pi }\left( g^{\ast }\right) E\left[ \otimes
_{i\in b_{1}}a_{l_{i}.}\right] \otimes \cdot \cdot \cdot \otimes E\left[
\otimes _{i\in b_{k}}a_{l_{i}.}\right] \\
&=&\int_{SO\left( 3\right) }\Delta _{l_{1}...l_{n}}^{\pi }\left( g\right) E
\left[ \otimes _{i\in b_{1}}a_{l_{i}.}\right] \otimes \cdot \cdot \cdot
\otimes E\left[ \otimes _{i\in b_{k}}a_{l_{i}.}\right] dg \\
&=&\Delta _{l_{1}...l_{n}}^{\pi }E\left[ \otimes _{i\in b_{1}}a_{l_{i}.}%
\right] \otimes \cdot \cdot \cdot \otimes E\left[ \otimes _{i\in
b_{k}}a_{l_{i}.}\right] .
\end{eqnarray*}

\noindent To prove the second part of the statement, suppose that $T(x)$
verifies $\sup_{x}\left( E\left\vert T(x)\right\vert ^{n}\right) <\infty $,
and that its associated harmonic coefficients verify (\ref{JC2}). Then, for
every fixed rotation $g^{\ast }\in SO(3),$
\begin{eqnarray*}
&&\sum_{\pi =\left\{ b_{1},...,b_{k}\right\} \in \Pi \left( n\right) }\left(
-1\right) ^{k-1}\left( k-1\right) !E\left[ \otimes _{i\in
b_{1}}D^{l_{i}}\left( g^{\ast }\right) a_{l_{i}.}\right] \otimes \cdot \cdot
\cdot \otimes E\left[ \otimes _{i\in b_{k}}D^{l_{i}}\left( g^{\ast }\right)
a_{l_{i}.}\right] \\
&=&\sum_{\pi =\left\{ b_{1},...,b_{k}\right\} \in \Pi \left( n\right)
}\left( -1\right) ^{k-1}\left( k-1\right) !\times \\
&&\times \lbrack D^{l_{v^{\pi }\left( 1\right) }}\left( g^{\ast }\right)
\otimes \cdot \cdot \cdot \otimes D^{l_{v^{\pi }\left( n\right) }}\left(
g^{\ast }\right) ]E\left[ \otimes _{i\in b_{1}}a_{l_{i}.}\right] \otimes
\cdot \cdot \cdot \otimes E\left[ \otimes _{i\in b_{k}}a_{l_{i}.}\right] \\
&=&\sum_{\pi =\left\{ b_{1},...,b_{k}\right\} \in \Pi \left( n\right)
}\left( -1\right) ^{k-1}\left( k-1\right) !\times \Delta
_{l_{1}...l_{n}}^{\pi }\left( g^{\ast }\right) E\left[ \otimes _{i\in
b_{1}}a_{l_{i}.}\right] \otimes \cdot \cdot \cdot \otimes E\left[ \otimes
_{i\in b_{k}}a_{l_{i}.}\right] \\
&=&\sum_{\pi =\left\{ b_{1},...,b_{k}\right\} \in \Pi \left( n\right)
}\left( -1\right) ^{k-1}\left( k-1\right) !E\left[ \otimes _{i\in
b_{1}}a_{l_{i}.}\right] \otimes \cdot \cdot \cdot \otimes E\left[ \otimes
_{i\in b_{k}}a_{l_{i}.}\right]
\end{eqnarray*}%
By the definition of cumulants, this last equality gives that
\begin{equation*}
E\left[ D^{l_{1}}(g^{\ast })a_{l_{1}.}\otimes \cdot \cdot \cdot \otimes
D^{l_{n}}(g^{\ast })a_{l_{n}.}\right] =E\left[ a_{l_{1}.}\otimes \cdot \cdot
\cdot \otimes a_{l_{n}.}\right] .
\end{equation*}%
Since $g^{\ast }$ is arbitrary, the $n$-weak isotropy follows from (\ref%
{specrap}).
\end{proof}

\bigskip

\textbf{Remark. }By combining (\ref{Spectre3}) and (\ref{johnChin}) we
obtain for instance that the $n$th cumulant polyspectrum of an isotropic
field verifies the identity%
\begin{eqnarray*}
S_{l_{1}...l_{n}}^{c} &=&\Delta _{l_{1}...l_{n}}S_{l_{1}...l_{n}}^{c} \\
&=&\sum_{\pi =\left\{ b_{1},...,b_{k}\right\} \in \Pi \left( n\right)
}\left( -1\right) ^{k-1}\left( k-1\right) !\Delta _{l_{1}...l_{n}}^{\pi }E
\left[ \otimes _{i\in b_{1}}a_{l_{i}.}\right] \otimes \cdot \cdot \cdot
\otimes E\left[ \otimes _{i\in b_{k}}a_{l_{i}.}\right] .
\end{eqnarray*}

\section{Angular polyspectra and the structure of $\Delta _{l_{1}...l_{n}} %
\label{S : APdelta}$}

\subsection{Spectra of strongly isotropic fields \label{SS : Spectra}}

Our aim in this section is to investigate more deeply the structure of the
matrix $\Delta _{l_{1}...l_{n}}$ appearing in (\ref{not1}), in order to
derive an explicit characterization for the angular polyspectra. As a
preliminary example, we deal with the case $n=2$.

\begin{proposition}
\label{P : Poly n=2}For integers $l_{1},l_{2}\geq 0$, one has that%
\begin{equation}
\Delta _{l_{1}l_{2}}=\int_{SO(3)}\left\{ D^{l_{1}}(g)\otimes
D^{l_{2}}(g)\right\} dg=\delta
_{l_{1}}^{l_{2}}C_{l_{1}.l_{2}.}^{00}(C_{l_{1}.l_{2}.}^{00})^{\prime }\text{,%
}  \label{marion}
\end{equation}%
that is: if $l_{1}\neq l_{2}$, then $\Delta _{l_{1}l_{2}}$ is a $\left(
2l_{1}+1\right) \left( 2l_{2}+1\right) \times \left( 2l_{1}+1\right) \left(
2l_{2}+1\right) $ zero matrix; if $l_{1}=l_{2}$, then $\Delta
_{l_{1}l_{2}}=\Delta _{l_{1}l_{1}}$ is given by $%
C_{l_{1}.l_{1}.}^{00}(C_{l_{1}.l_{1}.}^{00})^{\prime }$.
\end{proposition}

\begin{proof}
Using the equivalence of the two representations $D^{l_{1}}(g)\otimes
D^{l_{2}}(g)$ and $\oplus _{\lambda =|l_{2}-l_{1}|}^{l_{2}+l_{1}}D^{\lambda
}(g)$, as well as the definition of the Clebsch-Gordan matrices, we obtain
that%
\begin{equation}
\int_{SO(3)}\left\{ D^{l_{1}}(g)\otimes D^{l_{2}}(g)\right\}
dg=C_{l_{1}l_{2}}\left[ \int_{SO(3)}\left\{ \oplus _{\lambda
=|l_{2}-l_{1}|}^{l_{2}+l_{1}}D^{\lambda }(g)\right\} dg\right]
C_{l_{1}l_{2}}^{\ast }\text{.}  \label{fleet}
\end{equation}%
Now, if $l_{1}\neq l_{2}$, then the RHS of (\ref{fleet}) is equal to the
zero matrix since, as a consequence of the Peter-Weyl theorem and for $%
\lambda \neq 0$, the entries of $D^{\lambda }(\cdot )$ are orthogonal to the
constants. If $l_{1}=l_{2}$, then the integrated matrix on the RHS of (\ref%
{fleet}) becomes $\int_{SO(3)}\left\{ \oplus _{\lambda
=0}^{2l_{1}}D^{\lambda }(g)\right\} dg$, that is, a $\left( 2l_{1}+1\right)
^{2}\times \left( 2l_{1}+1\right) ^{2}$ matrix which is zero everywhere,
except for the entry in the top-left corner, which is equal to one (since $%
\int_{SO\left( 3\right) }dg=1$). The proof is concluded by checking that%
\begin{equation*}
C_{l_{1}l_{1}}\left[ \int_{SO(3)}\left\{ \oplus _{\lambda
=0}^{2l_{1}}D^{\lambda }(g)\right\} dg\right] C_{l_{1}l_{1}}^{\ast
}=C_{l_{1}.l_{1}.}^{00}(C_{l_{1}.l_{1}.}^{00})^{\prime }.
\end{equation*}
\end{proof}

\bigskip

\textbf{\noindent Remark. }Recall that $C_{l_{1}.l_{2}.}^{00}$ is a \textsl{%
column} vector of dimension $\left( 2l_{1}+1\right) \left( 2l_{2}+1\right) $%
, corresponding to the first column of the matrix $C_{l_{1}l_{2}}$. Also,
according e.g. to \cite[formula 8.5.1.1]{VMK}, one has that
\begin{equation*}
C_{l_{1}.l_{2}.}^{00}=\left\{ \frac{(-1)^{m_{1}}}{2l_{1}+1}\delta
_{l_{1}}^{l_{2}}\delta _{m_{1}}^{-m_{2}}\right\}
_{m_{1}=-l_{1},...,l_{1};m_{2}=-l_{2},...,l_{2}}\text{.}
\end{equation*}

\bigskip

Proposition \ref{P : Poly n=2} provides a characterization of the spectrum
of a strongly isotropic field.

\begin{corollary}
Let $T$ be a strongly isotropic field with second moments, and let the
vectors of the harmonic coefficients $\left\{ a_{l.}\right\} $ be defined
according to (\ref{specrap}). Then, for any integers $l_{1},l_{2}\geq 0$,
one has that%
\begin{equation}
E\left\{ a_{l_{1}.}\otimes a_{l_{2}.}\right\} =\left\{ \frac{(-1)^{m_{1}}}{%
2l_{1}+1}\delta _{l_{1}}^{l_{2}}\delta _{m_{1}}^{-m_{2}}C_{l_{1}}\right\}
\text{ }  \label{foxes}
\end{equation}%
for some $C_{l_{1}}\geq 0$ depending uniquely on $l_{1}$.
\end{corollary}

\begin{proof}
According to (\ref{crurel}), one has that%
\begin{equation*}
E\left\{ a_{l_{1}.}\otimes a_{l_{2}.}\right\} =\delta
_{l_{1}}^{l_{2}}C_{l_{1}.l_{2}.}^{00}(C_{l_{1}.l_{2}.}^{00})^{\prime
}E\left\{ a_{l_{1}.}\otimes a_{l_{2}.}\right\} \text{,}
\end{equation*}%
implying that $E\left\{ a_{l_{1}.}\otimes a_{l_{2}.}\right\} $ is (a) equal
to the zero vector for $l_{1}\neq l_{2}$, and (b) of the form $%
C_{l_{1}}\times C_{l_{1}.l_{2}.}^{00},$ for some constant $C_{l_{1}}$, when $%
l_{1}=l_{2}$. To see that $C_{l_{1}}$ cannot be negative, just observe that $%
a_{l_{1}0}$ is real-valued for every $l_{1}\geq 0$, so that (\ref{foxes})
yields that%
\begin{equation*}
C_{l_{1}}=\left( 2l_{1}+1\right) \times E\left( a_{l_{1}0}^{2}\right) .
\end{equation*}
\end{proof}

In the subsequent two subsections, we shall obtain, for every $n\geq 3$, a
characterization of $\Delta _{l_{1}...,l_{n}}$ and $E\{a_{l_{1}.}\otimes $ $%
\cdot \cdot \cdot \otimes a_{l_{n}.}\}$, respectively analogous to (\ref%
{marion}) and (\ref{foxes}).

\subsection{The structure of $\Delta _{l_{1}...l_{n}}\label{SS : Structue}$}

We first need to establish some further notation.

\bigskip

\textbf{\noindent Definition B. }Fix $n\geq 3$. For integers $%
l_{1},...,l_{n}\geq 0$, we define $C_{l_{1}...l_{n}}$ to be the unitary
matrix, of dimension
\begin{equation*}
\prod_{j=1}^{n}\left( 2l_{j}+1\right) \times \prod_{j=1}^{n}\left(
2l_{j}+1\right) \text{,}
\end{equation*}%
connecting the following two equivalent representations of $SO\left(
3\right) $%
\begin{equation}
D^{l_{1}}(.)\otimes D^{l_{2}}(.)\otimes \cdot \cdot \cdot \otimes
D^{l_{n}}(.)  \label{his}
\end{equation}%
and
\begin{equation}
\oplus _{\lambda _{1}=|l_{2}-l_{1}|}^{l_{2}+l_{1}}\oplus _{\lambda
_{2}=|l_{3}-\lambda _{1}|}^{l_{3}+\lambda _{1}}...\oplus _{\lambda
_{n-1}=|l_{n}-\lambda _{n-2}|}^{l_{n}+\lambda _{n-2}}D^{\lambda _{n-1}}(.)%
\text{.}  \label{name}
\end{equation}

$\bigskip $

\textbf{\noindent Remarks.} (1) Fix $l_{1},...,l_{n}\geq 0$, as well as $%
g\in SO\left( 3\right) $. Then, the matrix%
\begin{equation}
\oplus _{\lambda _{1}=|l_{2}-l_{1}|}^{l_{2}+l_{1}}\oplus _{\lambda
_{2}=|l_{3}-\lambda _{1}|}^{l_{3}+\lambda _{1}}...\oplus _{\lambda
_{n-1}=|l_{n}-\lambda _{n-2}|}^{l_{n}+\lambda _{n-2}}D^{\lambda _{n-1}}(g)
\label{remark}
\end{equation}%
is a block-diagonal matrix, obtained as follows. (a) Consider vectors of
integers $\left( \lambda _{1},...,\lambda _{n-1}\right) $ satisfying the
relations $\left\vert l_{2}-l_{1}\right\vert \leq \lambda _{1}\leq
l_{1}+l_{2}$, and $\left\vert l_{k+1}-\lambda _{k-1}\right\vert \leq \lambda
_{k}\leq l_{k+1}+\lambda _{k-1}$, for $k=2,...,n-1$. (b) Introduce a (total)
order $\prec _{0}$ on the collection of these vectors by saying that
\begin{equation}
\left( \lambda _{1},...,\lambda _{n-1}\right) \prec _{0}\left( \lambda
_{1}^{\prime },...,\lambda _{n-1}^{\prime }\right) \text{,}  \label{order}
\end{equation}%
whenever either $\lambda _{1}<\lambda _{1}^{\prime }$, or there exists $%
k=2,...,n-2$ such that $\lambda _{j}=\lambda _{j}^{\prime }$ for every $%
j=1,...,k$, and $\lambda _{k+1}<\lambda _{k+1}^{\prime }$. (c) Associate to
each vector $\left( \lambda _{1},...,\lambda _{n-1}\right) $ the matrix $%
D^{\lambda _{n-1}}\left( g\right) $. (d) Construct a block-diagonal matrix
by disposing the matrices $D^{\lambda _{n-1}}\left( g\right) $ from the
top-left corner to the bottom-right corner, in increasing order with respect
to $\prec _{0}$. As an example, consider the case where $n=3$ and $%
l_{1}=l_{2}=l_{3}=1$. Here, the vectors $\left( \lambda _{1},\lambda
_{2}\right) $ involved in the direct sum (\ref{name}) are (in increasing
order with respect to $\prec _{0}$)%
\begin{equation*}
\left( 0,1\right) ,\text{ }\left( 1,0\right) ,\text{ }\left( 1,1\right) ,%
\text{ }\left( 1,2\right) ,\text{ }\left( 2,1\right) ,\text{ }\left(
2,2\right) \text{ \ and \ }\left( 2,3\right) ,
\end{equation*}%
and the matrix (\ref{remark}) is therefore given by
\begin{equation}
\left(
\begin{array}{ccccccc}
D^{1}\left( g\right) & ... & ... & ... & ... & ... & ... \\
... & 1 & ... & ... & ... & ... & ... \\
... & ... & D^{1}\left( g\right) & ... & ... & ... & ... \\
... & ... & ... & D^{2}\left( g\right) & ... & ... & ... \\
... & ... & ... & ... & D^{1}\left( g\right) & ... & ... \\
... & ... & ... & ... & ... & D^{2}\left( g\right) & ... \\
... & ... & ... & ... & ... & ... & D^{3}\left( g\right)%
\end{array}%
\right)  \label{matriciona}
\end{equation}%
where the dots indicate \textsl{zero} entries, and we have used the fact
that $D^{0}\left( g\right) \equiv 1$.

(2) The fact that the representation (\ref{name}) has dimension $%
\prod_{j=1}^{n}\left( 2l_{j}+1\right) $ is a direct consequence of formula (%
\ref{fdfd}).

(3) The fact that the two representations (\ref{his}) and (\ref{name}) are
equivalent can be proved by iteration. Indeed, by standard representation
theory, on has that (\ref{his}) is equivalent to
\begin{equation*}
\oplus _{\lambda _{1}=|l_{2}-l_{1}|}^{l_{2}+l_{1}}D^{\lambda _{1}}(.)\otimes
D^{l_{3}}\left( \cdot \right) \otimes \cdot \cdot \cdot \otimes
D^{l_{n}}\left( \cdot \right) \text{,}
\end{equation*}%
which is in turn equivalent to
\begin{equation*}
\oplus _{\lambda _{1}=|l_{2}-l_{1}|}^{l_{2}+l_{1}}\oplus _{\lambda
_{2}=|l_{3}-\lambda _{1}|}^{l_{3}+\lambda _{1}}D^{\lambda _{2}}(.)\otimes
D^{l_{4}}\left( \cdot \right) \otimes \cdot \cdot \cdot \otimes
D^{l_{n}}\left( \cdot \right) \text{.}
\end{equation*}%
By iterating the same procedure until all tensor products have disappeared
(that is, by successively replacing the tensor product $D^{\lambda
_{k}}(.)\otimes D^{l_{k+2}}\left( \cdot \right) $ with $\oplus _{\lambda
_{k+1}=|l_{k+2}-\lambda _{k}|}^{l_{k+2}+\lambda _{k}}D^{\lambda _{2}}(.)$
for $k=2,...,n-1$), one obtains the desired conclusion.

\bigskip

For every $n\geq 3$ and every $l_{1},...,l_{n}\geq 0$, the elements of the
matrix $C_{l_{1}...l_{n}}$, introduced in Definition B, can be written in
the form $C_{l_{1}m_{1}...l_{n}m_{n}}^{\lambda _{1}...\lambda _{n-1},\mu
_{n-1}}$. The indices $\left( m_{1},...,m_{n}\right) $ are such that $%
-l_{i}\leq m_{i}\leq l_{i}$ ($i=1,...,n$) and label rows; on the other hand,
the indices $\left( \lambda _{1}...\lambda _{n-1},\mu _{n-1}\right) $ label
columns, and verify the relations $\left\vert l_{2}-l_{1}\right\vert \leq
\lambda _{1}\leq l_{1}+l_{2}$, $\left\vert l_{k+1}-\lambda _{k-1}\right\vert
\leq \lambda _{k}\leq l_{k+1}+\lambda _{k-1}$ ($k=2,...,n-1$) and $-\lambda
_{n-1}\leq \mu _{n-1}\leq \lambda _{n-1}$. It is well known (see e.g. \cite%
{VMK}) that the quantity $C_{l_{1}m_{1}...l_{n}m_{n}}^{\lambda
_{1}...\lambda _{n-1},\mu _{n-1}}$ can be represented as a \textsl{%
convolution }of the Clebsch-Gordan coefficients introduced in Section \ref%
{SS : CGmat}, namely:%
\begin{eqnarray*}
C_{l_{1}m_{1}...l_{n}m_{n}}^{\lambda _{1},...,\lambda _{n-1},\mu _{n-1}}
&=&C_{l_{1}m_{1}...l_{n-1}m_{n-1}}^{\lambda _{1},...,\lambda
_{n-2},.}C_{\lambda _{n-2}l_{n}m_{n}}^{\lambda _{n-1}\mu _{n-1}} \\
&=&\sum_{\mu _{n-2}}\left\{ \sum_{\mu _{1}...\mu
_{n-3}}C_{l_{1}m_{1}l_{2}m_{2}}^{\lambda _{1}\mu _{1}}C_{\lambda _{1}\mu
_{1}l_{3}m_{3}}^{\lambda _{2}\mu _{2}}...C_{\lambda _{n-3}\mu
_{n-3}l_{n-1}m_{n-1}}^{\lambda _{n-2}\mu _{n-2}}\right\} C_{\lambda
_{n-2}\mu _{n-2}l_{n}m_{n}}^{\lambda _{n-1}\mu _{n-1}} \\
&=&\sum_{\mu _{1}...\mu _{n-2}}C_{l_{1}m_{1}l_{2}m_{2}}^{\lambda _{1}\mu
_{1}}C_{\lambda _{1}\mu _{1}l_{3}m_{3}}^{\lambda _{2}\mu _{2}}...C_{\lambda
_{n-3}\mu _{n-3}l_{n-1}m_{n-1}}^{\lambda _{n-2}\mu _{n-2}}C_{\lambda
_{n-2}\mu _{n-2}l_{n}m_{n}}^{\lambda _{n-1}\mu _{n-1}}\text{.}
\end{eqnarray*}

\bigskip

\textbf{\noindent Remark. }Given an enumeration of the coefficients $%
C_{l_{1}m_{1}...l_{n}m_{n}}^{\lambda _{1}...\lambda _{n-1},\mu _{n-1}}$, the
matrix $C_{l_{1}...l_{n}}$ can be built (analogously to the case of the
Clebsch-Gordan matrices of Section \ref{SS : CGmat}) by disposing rows (from
top to bottom) and columns (from left to right) increasingly according to
two separate total orders. The order $\prec _{r}$ on the symbols $\left(
m_{1},...,m_{n}\right) $ is obtained by setting that $\left(
m_{1},...,m_{n}\right) \prec _{r}\left( m_{1}^{\prime },...,m_{n}^{\prime
}\right) $ whenever either $m_{1}<m_{1}^{\prime }$, or there exists $%
k=2,...,n-1$ such that $m_{j}=m_{j}^{\prime }$ for every $j=1,...,k$, and $%
m_{k+1}<m_{k+1}^{\prime }$. The order $\prec _{c}$ on the symbols $\left(
\lambda _{1}...\lambda _{n-1},\mu _{n-1}\right) $ is obtained by setting
that $\left( \lambda _{1}...\lambda _{n-1},\mu _{n-1}\right) \prec
_{c}\left( \lambda _{1}^{\prime }...\lambda _{n-1}^{\prime },\mu
_{n-1}^{\prime }\right) $ whenever either $\left( \lambda _{1},...,\lambda
_{n-1}\right) \prec _{0}\left( \lambda _{1}^{\prime },...,\lambda
_{n-1}^{\prime }\right) $, as defined in (\ref{order}), or $\lambda
_{i}=\lambda _{i}^{\prime }$ for every $i=1,...,n-1$ and $\mu _{n-1}<\mu
_{n-1}^{\prime }$.

\bigskip

One has also the following (useful) alternative representation of
generalized Clebsch-Gordan matrices.

\begin{proposition}
\label{P : ALternative GenCG}For every $n\geq 3$ and every $%
l_{1},...,l_{n}\geq 0$, one can represent the matrix $C_{l_{1}...l_{n}}$, as
follows%
\begin{equation*}
C_{l_{1}...l_{n}}=\left\{ C_{l_{1}l_{2}l_{3}...l_{n-1}}\otimes
I_{2l_{n}+1}\right\} \left\{ (\oplus _{\lambda
_{1}=|l_{2}-l_{1}|}^{l_{2}+l_{1}}...\oplus _{\lambda _{n-2}=|l_{n}-\lambda
_{n-3}|}^{l_{n}+\lambda _{n-3}}C_{\lambda _{n-2}l_{n}}\right\} \text{,}
\end{equation*}%
where $I_{m}$ indicates a $m\times m$ identity matrix. Also, one has that%
\begin{eqnarray*}
C_{l_{1}...l_{n}} &=&(C_{l_{1}l_{2}}\otimes I_{2l_{3}+1}\otimes ...\otimes
I_{2l_{n}+1})\times \left[ (\oplus _{\lambda
=|l_{2}-l_{1}|}^{l_{2}+l_{1}}C_{\lambda l_{3}})\otimes ...\otimes
I_{2l_{n}+1}\right] \\
&&\times ...\times \left[ (\oplus _{\lambda
_{1}=|l_{2}-l_{1}|}^{l_{2}+l_{1}}...\oplus _{\lambda _{n-2}=|l_{n}-\lambda
_{n-3}|}^{l_{n}+\lambda _{n-3}}C_{\lambda _{n-2}l_{n}}\right] \text{,}
\end{eqnarray*}%
where $\times $ stands for the usual product between matrices.
\end{proposition}

\bigskip

\textbf{Definition C}. For every $n\geq 3$ and every $l_{1},...,l_{n}\geq 0$%
, we define $E_{l_{1}...l_{n}}$ to be the $\Pi _{j=1}^{n}\left(
2l_{j}+1\right) $ $\times $ $\Pi _{j=1}^{n}\left( 2l_{j}+1\right) $ square
matrix
\begin{equation}
E_{l_{1}...l_{n}}:=\oplus _{\lambda
_{1}=|l_{2}-l_{1}|}^{l_{2}+l_{1}}...\oplus _{\lambda _{n-1}=|l_{n}-\lambda
_{n-2}|}^{l_{n}+\lambda _{n-2}}\delta _{\lambda _{n-1}}^{0}I_{2\lambda
_{n-1}+1}\text{.}  \label{matE}
\end{equation}%
In other words, $E_{l_{1}...l_{n}}$ is the diagonal matrix built from the
matrix (\ref{remark}), by replacing every block of the type $D^{\lambda
_{n-1}}\left( g\right) $, with $\lambda _{n-1}>0$, with a $\left( 2\lambda
_{n-1}+1\right) \times \left( 2\lambda _{n-1}+1\right) $ zero matrix, and by
letting the $1\times 1$ blocks $D^{0}\left( g\right) =1$ unchanged. For
instance, by setting $n=3$ and $l_{1}=l_{2}=l_{3}=1$ (and by using (\ref%
{matriciona})) one obtains a $27\times 27$ matrix $E_{111}$ whose entries
are all zero, except for the fourth element (starting from the top-left
corner) of the main diagonal.

\bigskip

The following result states that the matrix $\Delta _{l_{1}...l_{n}}$ can be
diagonalized in terms of $C_{l_{1}...l_{n}}$ and $E_{l_{1}...l_{n}}$.

\begin{proposition}
\label{propb}The matrix $\Delta _{l_{1}...l_{n}}$ can be diagonalized as
\begin{equation}
\Delta
_{l_{1}...l_{n}}=C_{l_{1}...l_{n}}E_{l_{1}...l_{n}}C_{l_{1}...l_{n}}^{\ast }%
\text{ ,}  \label{form2}
\end{equation}%
where $E_{l_{1}...l_{n}}$ is the matrix introduced in Definition C.
\end{proposition}

\begin{proof}
One has that
\begin{eqnarray}
\Delta _{l_{1}...l_{n}} &=&\int_{SO\left( 3\right) }D^{l_{1}}\left( g\right)
\otimes D^{l_{2}}\left( g\right) \otimes \cdot \cdot \cdot \otimes
D^{l_{n}}\left( g\right) dg  \label{u} \\
&=&\int_{SO\left( 3\right) }\left[ C_{l_{1}...l_{n}}\oplus _{\lambda
_{1}=|l_{2}-l_{1}|}^{l_{2}+l_{1}}\oplus _{\lambda _{2}=|l_{3}-\lambda
_{1}|}^{l_{3}+\lambda _{1}}...\oplus _{\lambda _{n-1}=|l_{n}-\lambda
_{n-2}|}^{l_{n}+\lambda _{n-2}}D^{\lambda _{n-1}}(g)C_{l_{1}...l_{n}}^{\ast }%
\right] dg.  \notag
\end{eqnarray}%
By linearity and by the definition of the integral of a matrix-valued
function, one has that the last line of (\ref{u}) equals
\begin{equation*}
C_{l_{1}...l_{n}}\left[ \oplus _{\lambda
_{1}=|l_{2}-l_{1}|}^{l_{2}+l_{1}}\oplus _{\lambda _{2}=|l_{3}-\lambda
_{1}|}^{l_{3}+\lambda _{1}}...\oplus _{\lambda _{n-1}=|l_{n}-\lambda
_{n-2}|}^{l_{n}+\lambda _{n-2}}\int_{SO\left( 3\right) }D^{\lambda
_{n-1}}(g)dg\right] \text{ }C_{l_{1}...l_{n}}^{\ast }\text{.}
\end{equation*}%
Now observe that, if $\lambda _{n-1}>0$, then $\int_{SO\left( 3\right)
}D^{\lambda _{n-1}}(g)dg$ equals a $\left( 2\lambda _{n-1}+1\right) \times
\left( 2\lambda _{n-1}+1\right) $ zero matrix, whereas $\int_{SO\left(
3\right) }D^{0}(g)dg=\int_{SO\left( 3\right) }1dg=1$. The conclusion is
obtained by resorting to the definition of $E_{l_{1}...l_{n}}$ given in (\ref%
{matE}).
\end{proof}

\subsection{Existence and characterization of reduced polyspectra of
arbitrary orders}

Combining the previous Proposition with (\ref{propa}), we obtain the main
result of this paper.

\begin{theorem}
\label{taqqu}If a random field is strongly isotropic with finite moments of
order $n\geq 3$, then for every $l_{1},...,l_{n}$ there exists two arrays $%
P_{l_{1}....l_{n}}(\lambda _{1},...,\lambda _{n-3})$ and $%
P_{l_{1}....l_{n}}^{C}(\lambda _{1},...,\lambda _{n-3}),$ with $%
|l_{2}-l_{1}|\leq \lambda _{1}\leq l_{2}+l_{1},$ $|l_{3}-\lambda _{1}|\leq
\lambda _{2}\leq l_{3}+\lambda _{1},...,|l_{n-2}-\lambda _{n-4}|$ $\leq
\lambda _{n-3}\leq l_{n-2}+\lambda _{n-4},$ such that%
\begin{eqnarray}
Ea_{l_{1}m_{1}}...a_{l_{n}m_{n}} &=&(-1)^{m_{n}}\sum_{\lambda
_{1}=l_{2}-l_{1}}^{l_{2}+l_{1}}...\sum_{\lambda
_{n-3}}C_{l_{1}m_{1}....l_{n-1}m_{n-1}}^{\lambda _{1}...\lambda
_{n-3}l_{n}-m_{n}}P_{l_{1}....l_{n}}(\lambda _{1},...,\lambda _{n-3})\text{ }
\label{Spectre1} \\
\mathrm{Cum}\left\{ a_{l_{1}m_{1}},...,a_{l_{n}m_{n}}\right\}
&=&(-1)^{m_{n}}\sum_{\lambda _{1}=l_{2}-l_{1}}^{l_{2}+l_{1}}...\sum_{\lambda
_{n-3}}C_{l_{1}m_{1}....l_{n-1}m_{n-1}}^{\lambda _{1}...\lambda
_{n-3}l_{n}-m_{n}}P_{l_{1}....l_{n}}^{C}(\lambda _{1},...,\lambda _{n-3})
\label{Spectre11/2} \\
C_{l_{1}m_{1}....l_{n-1}m_{n-1}}^{\lambda _{1}...\lambda _{n-3};l_{n}-m_{n}}
&=&\sum_{\mu _{1}}...\sum_{\mu _{n-3}}C_{l_{1}m_{1}l_{2}m_{2}}^{\lambda
_{1}\mu _{1}}C_{\lambda _{1}\mu _{1}l_{3}m_{3}}^{\lambda _{2}\mu
_{2}}...C_{\lambda _{n-3}\mu _{n-3}l_{n-1}m_{n-1}}^{l_{n},-m_{n}}\text{ .}
\label{Spectre2}
\end{eqnarray}
\end{theorem}

\bigskip

\noindent \textbf{Remark. }For a fixed $n\geq 2$, the real-valued arrays $%
\left\{ P_{l_{1}...l_{n}}\left( \cdot \right) :l_{1},...,l_{n}\geq 0\right\}
$ and $\left\{ P_{l_{1}...l_{n}}^{C}\left( \cdot \right)
:l_{1},...,l_{n}\geq 0\right\} $ are, respectively, the \textsl{reduced
polyspectrum of order }$n-1$ and the \textsl{reduced cumulant polyspectrum
of order }$n-1$ associated with the underlying strongly isotropic random
field.

\bigskip

\begin{proof}[Proof of Theorem \ref{taqqu}]
We shall prove only (\ref{Spectre1}), since the proof of (\ref{Spectre11/2})
is entirely analogous. By Proposition \ref{propa} and Proposition \ref{propb}%
, if the random field is isotropic, then
\begin{equation*}
S_{l_{1}...l_{n}}=C_{l_{1}...l_{n}}E_{l_{1}...l_{n}}C_{l_{1}...l_{n}}^{\ast
}S_{l_{1}...l_{n}}\text{ ,}
\end{equation*}%
that is, because $C_{l_{1}...l_{n}}$ is unitary
\begin{equation*}
C_{l_{1}...l_{n}}^{\ast
}S_{l_{1}...l_{n}}=E_{l_{1}...l_{n}}C_{l_{1}...l_{n}}^{\ast
}S_{l_{1}...l_{n}}\text{ .}
\end{equation*}%
It follows that $S_{l_{1}...l_{n}}$ is a solution if and only if the column
vector $C_{l_{1}...l_{n}}^{\ast }S_{l_{1}...l_{n}}$ has zeroes corresponding
to the zeroes of $E_{l_{1}...l_{n}},$ whereas the elements corresponding to
unity can be arbitrary. In view of the orthonormality properties of $%
C_{l_{1}...l_{n}}^{\ast },$ this condition is met if, and only if, $%
S_{l_{1}...l_{n}}$ is a linear combination of the columns in the matrix $%
C_{l_{1}...l_{n}}^{\ast }$ corresponding to non-zero elements of the
diagonal $E_{l_{1}...l_{n}}.$ These linear combinations can be written
explicitly as
\begin{eqnarray*}
&&\sum_{\lambda _{1}=l_{2}-l_{1}}^{l_{2}-l_{1}}\sum_{\lambda
_{2}=l_{3}-\lambda _{1}}^{l_{3}+\lambda _{1}}...\sum_{\lambda
_{n-1}=l_{n}-\lambda _{n-2}}^{l_{n}+\lambda
_{n-2}}C_{l_{1}m_{1}....l_{n}m_{n}}^{\lambda _{1}...\lambda _{n-2}lm}%
\widetilde{P}_{l_{1}....l_{n}}(\lambda _{1},...,\lambda _{n-3},\lambda
_{n-2})\delta _{l}^{0} \\
&=&\sum_{\lambda _{1}=l_{2}-l_{1}}^{l_{2}-l_{1}}\sum_{\lambda
_{2}=l_{3}-\lambda _{1}}^{l_{3}+\lambda _{1}}...\sum_{\lambda
_{n-1}=l_{n}-\lambda _{n-2}}^{l_{n}+\lambda _{n-2}}\left\{ \sum_{\mu
_{1}...\mu _{n-2}}C_{l_{1}m_{1}l_{2}m_{2}}^{\lambda _{1}\mu _{1}}C_{\lambda
_{1}\mu _{1}l_{3}m_{3}}^{\lambda _{2}\mu _{2}}...C_{\lambda _{n-2}\mu
_{n-2}.l_{n}m_{n}}^{lm}\delta _{l}^{0}\right\} \times \\
&&\text{ \ \ \ \ \ \ \ \ \ \ \ \ \ \ \ \ \ \ \ \ \ \ \ \ \ \ \ \ \ \ \ \ \ \
\ \ \ \ \ \ \ \ \ \ \ \ \ \ \ \ \ \ \ \ \ \ \ \ \ \ \ \ \ \ \ \ \ \ \ \ \ \
\ \ \ \ \ \ \ \ \ }\times \widetilde{P}_{l_{1}....l_{n}}(\lambda
_{1},...,\lambda _{n-3},\lambda _{n-2}) \\
&=&\sum_{\lambda _{1}=l_{2}-l_{1}}^{l_{2}-l_{1}}\sum_{\lambda
_{2}=l_{3}-\lambda _{1}}^{l_{3}+\lambda _{1}}...\sum_{\lambda
_{n-1}=l_{n}-\lambda _{n-2}}^{l_{n}+\lambda _{n-2}}\left\{ \sum_{\mu
_{1}...\mu _{n-2}}C_{l_{1}m_{1}l_{2}m_{2}}^{\lambda _{1}\mu _{1}}C_{\lambda
_{1}\mu _{1}l_{3}m_{3}}^{\lambda _{2}\mu _{2}}...C_{\lambda _{n-2}\mu
_{n-2}.l_{n}m_{n}}^{00}\right\} \times \\
&&\text{ \ \ \ \ \ \ \ \ \ \ \ \ \ \ \ \ \ \ \ \ \ \ \ \ \ \ \ \ \ \ \ \ \ \
\ \ \ \ \ \ \ \ \ \ \ \ \ \ \ \ \ \ \ \ \ \ \ \ \ \ \ \ \ \ \ \ \ \ \ \ \ \
\ \ \ \ \ \ \ \ \ \ }\times \widetilde{P}_{l_{1}....l_{n}}(\lambda
_{1},...,\lambda _{n-3},\lambda _{n-2}).
\end{eqnarray*}%
Recalling again that%
\begin{equation*}
C_{l_{1}m_{1}l_{2}m_{2}}^{0m}=\frac{(-1)^{m_{1}}}{2l_{1}+1}\delta
_{l_{1}}^{l_{2}}\delta _{m_{1}}^{-m_{2}}\delta _{m}^{0}\text{ ,}
\end{equation*}%
(see \cite{VMK}, 8.5.1.1), we obtain that
\begin{eqnarray*}
&=&\sum_{\lambda _{1}=l_{2}-l_{1}}^{l_{2}-l_{1}}\sum_{\lambda
_{2}=l_{3}-\lambda _{1}}^{l_{3}+\lambda _{1}}...\sum_{\lambda
_{n-1}=l_{n}-\lambda _{n-2}}^{l_{n}+\lambda _{n-2}}\left\{ \sum_{\mu
_{1}...\mu _{n-2}}C_{l_{1}m_{1}l_{2}m_{2}}^{\lambda _{1}\mu _{1}}C_{\lambda
_{1}\mu _{1}l_{3}m_{3}}^{\lambda _{2}\mu _{2}}...\frac{(-1)^{m_{n}}}{2l_{n}+1%
}\delta _{\lambda _{n-2}}^{l_{n}}\delta _{\mu _{n-2}}^{-m_{n}}\right\} \times
\\
&&\text{ \ \ \ \ \ \ \ \ \ \ \ \ \ \ \ \ \ \ \ \ \ \ \ \ \ \ \ \ \ \ \ \ \ \
\ \ \ \ \ \ \ \ \ \ \ \ \ \ \ \ \ \ \ \ \ \ \ \ \ \ \ \ \ \ \ \ \ \ \ \ \ \
\ \ \ \ \ \ \ \ \ \ \ }\times \widetilde{P}_{l_{1}....l_{n}}(\lambda
_{1},...,\lambda _{n-3},\lambda _{n-2}) \\
&=&\sum_{\lambda _{1}=l_{2}-l_{1}}^{l_{2}-l_{1}}\sum_{\lambda
_{2}=l_{3}-\lambda _{1}}^{l_{3}+\lambda _{1}}...\sum_{\lambda
_{n-1}=l_{n}-\lambda _{n-2}}^{l_{n}+\lambda _{n-2}}\left\{ \sum_{\mu
_{1}...\mu _{n-2}}C_{l_{1}m_{1}l_{2}m_{2}}^{\lambda _{1}\mu _{1}}C_{\lambda
_{1}\mu _{1}l_{3}m_{3}}^{\lambda _{2}\mu _{2}}...C_{\lambda _{n-3}\mu
_{n-3}.l_{n-1}m_{n-1}}^{l_{n}-m_{n}}(-1)^{m_{n}}\right\} \times \\
&&\text{ \ \ \ \ \ \ \ \ \ \ \ \ \ \ \ \ \ \ \ \ \ \ \ \ \ \ \ \ \ \ \ \ \ \
\ \ \ \ \ \ \ \ \ \ \ \ \ \ \ \ \ \ \ \ \ \ \ \ \ \ \ \ \ \ \ \ \ \ \ \ \ \
\ \ \ \ \ \ \ \ \ \ \ \ \ \ \ \ \ }\times P_{l_{1}....l_{n}}(\lambda
_{1},...,\lambda _{n-3}) \\
&=&\sum_{\lambda _{1}=l_{2}-l_{1}}^{l_{2}-l_{1}}\sum_{\lambda
_{2}=l_{3}-\lambda _{1}}^{l_{3}+\lambda _{1}}...\sum_{\lambda
_{n-1}=l_{n}-\lambda _{n-2}}^{l_{n}+\lambda
_{n-2}}C_{l_{1}m_{1}....l_{n-1}m_{n-1}}^{\lambda _{1}...\lambda
_{n-3}l_{n}-m_{n}}P_{l_{1}....l_{n}}(\lambda _{1},...,\lambda _{n-3}),
\end{eqnarray*}%
where we have set%
\begin{equation*}
P_{l_{1}....l_{n}}(\lambda _{1},...,\lambda _{n-3}):=\frac{1}{2l_{n}+1}%
\widetilde{P}_{l_{1}....l_{n}}(\lambda _{1},...,\lambda _{n-3},l_{n})\text{ .%
}
\end{equation*}

All there is left to show is that the coefficients of this linear
combination are necessarily real. To see this, it is sufficient to
specialize the previous discussion to the case where $%
m_{1}=m_{2}=...=m_{n}=0,$, and to observe that, in this case
\begin{equation*}
Ea_{l_{1}0}...a_{l_{n}0}=\sum_{\lambda _{1}}...\sum_{\lambda
_{n-3}}C_{l_{1}0....l_{n-1}0}^{\lambda _{1}...\lambda
_{n-3}l_{n}0}P_{l_{1}....l_{n}}(\lambda _{1},...,\lambda _{n-3})
\end{equation*}%
is real by definition (note indeed that the columns of $C_{l_{1}...l_{n}}$
are linearly independent).
\end{proof}

\bigskip

Let us illustrate the previous results by some more examples.

\bigskip

\textbf{Examples. }For $n=3,$ Theorem \ref{taqqu} implies that, under
isotropy
\begin{equation*}
Ea_{l_{1}m_{1}}a_{l_{2}m_{2}}a_{l_{3}m_{3}}=(-1)^{m_{3}}C_{l_{1}m_{1}l_{2}m_{2}}^{l_{3}-m_{3}}P_{l_{1}l_{2}l_{3}}%
\text{ .}
\end{equation*}%
From this last relation, we can recover the so-called reduced bispectrum,
noted $b_{l_{1}l_{2}l_{3}},$ defined for instance in \cite{Hu}, \cite{m2006}
and \cite{MarPTRF}, which satisfies indeed the relationship%
\begin{equation*}
P_{l_{1}l_{2}l_{3}}=b_{l_{1}l_{2}l_{3}}C_{l_{1}0l_{2}0}^{l_{3}0}\sqrt{\frac{%
(2l_{1}+1)(2l_{2}+1)}{(2l_{3}+1)4\pi }}\text{ .}
\end{equation*}%
For $n=4$ (i.e. the trispectrum, \cite{Hu}) we obtain the expression%
\begin{eqnarray*}
Ea_{l_{1}m_{1}}a_{l_{2}m_{2}}a_{l_{3}m_{3}}a_{l_{4}m_{4}}
&=&(-1)^{m_{4}}\sum_{\lambda
=|l_{2}-l_{1}|}^{l_{2}+l_{1}}C_{l_{1}m_{1}l_{2}m_{2}l_{3}m_{3}}^{\lambda
l_{4}-m_{4}}P_{l_{1}l_{2}l_{3}l_{4}}(\lambda ) \\
&=&\sum_{\lambda =|l_{2}-l_{1}|}^{l_{2}+l_{1}}\sum_{\mu =-\lambda }^{\lambda
}C_{l_{1}m_{1}l_{2}m_{2}}^{\lambda \mu }C_{\lambda \mu
l_{3}m_{3}}^{l_{4}-m_{4}}P_{l_{1}l_{2}l_{3}l_{4}}(\lambda )\text{ .}
\end{eqnarray*}

\bigskip

The next result gives a further probabilistic characterization of the
reduced bispectrum.

\begin{proposition}
Fix $n\geq 2$. A real-valued array $\left\{ A_{l_{1}...l_{n}}\left( \cdot
\right) :l_{1},...,l_{n}\geq 0\right\} $ is the reduced polyspectrum of
order $n-1$ (resp. the reduced cumulant polyspectrum of order $n-1$) of some
strongly isotropic random field if, and only if, there exists a sequence $%
\left\{ X_{l}:l\geq 0\right\} $ of zero-mean real-valued random variables
such that%
\begin{equation*}
\sum_{l\geq 0}\left( 2l+1\right) E\left[ X_{l}^{2}\right] <+\infty
\end{equation*}%
and, for every $l_{1},...,l_{n}\geq 0$%
\begin{equation}
E\left( X_{l_{1}}\cdot \cdot \cdot X_{l_{n}}\right) =\sum_{\lambda
_{1}=l_{2}-l_{1}}^{l_{2}+l_{1}}...\sum_{\lambda
_{n-3}}C_{l_{1}0....l_{n-1}0}^{\lambda _{1}...\lambda
_{n-3}l_{n}0}A_{l_{1}....l_{n}}(\lambda _{1},...,\lambda _{n-3})  \label{q}
\end{equation}%
(resp.%
\begin{equation}
\mathrm{Cum}\left\{ X_{l_{1}},\cdot \cdot \cdot ,X_{l_{n}}\right\}
=\sum_{\lambda _{1}=l_{2}-l_{1}}^{l_{2}+l_{1}}...\sum_{\lambda
_{n-3}}C_{l_{1}0....l_{n-1}0}^{\lambda _{1}...\lambda
_{n-3}l_{n}0}A_{l_{1}....l_{n}}(\lambda _{1},...,\lambda _{n-3}). \,
) \label{qq}
\end{equation}
\end{proposition}

\begin{proof}
We shall only prove (\ref{q}). For the necessity it is enough to take $%
X_{l}=a_{l0}$, where $a_{l0}$ is the harmonic coefficient of index $\left(
l,0\right) $ associated with a strongly isotropic field with moments of all
orders. For the sufficiency, we consider first the (anisotropic) random field%
\begin{equation*}
Z\left( x\right) =\sum_{l\geq 0}X_{l}Y_{l0}\left( x\right) .
\end{equation*}%
Then, by taking $T\left( x\right) =Z\left( gx\right) $, where $g$ is sampled
randomly with the uniform Haar measure on $SO\left( 3\right) $, one obtains
a random field with the desired characteristics.
\end{proof}

\bigskip

There are two very important issues that are left open by Theorem \ref{taqqu}%
. As a first issue, it seems natural to look for characterizations of the
reduced polyspectra $P_{l_{1}...l_{n}}$, at least under natural models of
physical interest. As a second point, we note that the explicit expressions
provided in Theorem \ref{taqqu} depend on the ordering $l_{1},...,l_{n}$ we
chose for the decomposition of $\Delta _{l_{1}...l_{n}}$. In the next two
sections, we try to address these (and other) points.

\section{Some Explicit Examples \label{examples}}

In this section we provide explicit computations for the reduced polyspectra
$P_{l_{1}...l_{n}}$ ($n\geq 2$), or $P_{l_{1}...l_{n}}^{C}$, for some models
of physical interest. Of course, the Gaussian isotropic fields can be easily
dealt with. Indeed, in this case one has that $P_{l_{1}...l_{n}}^{C}=0$ for
all $n\geq 3$. In what follows, we shall therefore be concerned with
polyspectra of \textsl{Gaussian subordinated} isotropic fields, that is,
random fields that can be written as a deterministic and non-linear function
of some collection of Gaussian isotropic fields. In general, this class of
random fields allow for a clear-cut mathematical treatment, whilst covering
a great array of empirically relevant circumstances.

\subsection{A simple physical model}

The general Gaussian-subordinated model has the form%
\begin{equation}
T=\sum_{j=1}^{q}f_{j}H_{j}\left( T_{G}/\sqrt{E\left( T_{G}^{2}\right) }%
\right) =f_{1}T_{G}+f_{2}(T_{G}^{2}/E\left( T_{G}^{2}\right) -1)+...,
\label{genmod}
\end{equation}%
where $f_{j}$ is a real constant, $H_{j}(.)$ denotes the $j$th Hermite
polynomial (see e.g. \cite{Surg}), and $T_{G}$ is a Gaussian, zero-mean
isotropic random field. Note that we have implicitly defined the sequence of
Hermite polynomials in such a way that $H_{1}\left( x\right) =x$, $%
H_{2}\left( x\right) =x^{2}-1$, $H_{3}\left( x\right) =x^{3}-3x$, and so on.
In this section, when no further specification is needed, the spectral
decomposition of the underlying Gaussian field $T_{G}$ is written%
\begin{equation*}
T_{G}=\sum_{lm}a_{lm}Y_{lm}.
\end{equation*}%
We shall sometimes use the following notation
\begin{eqnarray}
T &=&\sum_{lm}\widetilde{a}_{lm}Y_{lm}=\sum_{j=1}^{q}f_{j}a_{lm}(j)Y_{lm}%
\text{ , }  \label{tmoe} \\
a_{lm}(j) &=&\int_{S^{2}}H_{j}\left( T_{G}\left( x\right) /\sqrt{E\left(
T_{G}^{2}\right) }\right) \overline{Y_{lm}}\left( x\right) dx\text{,}
\label{tmoe2} \\
\widetilde{a}_{lm} &=&\sum_{j=1}^{q}a_{lm}\left( j\right) \text{. }
\label{tmoe3}
\end{eqnarray}

\bigskip

For instance, models of Cosmic Microwave Background radiation are currently
dominated by assumptions such as the Sachs-Wolfe model with the so-called
\textsl{Bardeen's potential} (see e.g. \cite{Bart} or \cite{dodelson}). The
latter can be written down explicitly as%
\begin{equation}
T=T_{G}+f_{NL}(T_{G}^{2}-ET_{G}^{2})\text{ ,}  \label{swbar}
\end{equation}%
where $f_{NL}$ is a nonlinearity parameters which depends upon physical
constants in the associated \textsl{\textquotedblleft
slow-roll\textquotedblright\ inflationary model}\emph{\ }(see e.g. \cite%
{Bart}). Note that (\ref{swbar}) has can be written in the form (\ref{genmod}%
), by setting $f_{1}=1$, $f_{2}=f_{NL}\times E\left( T_{G}^{2}\right) $ and $%
f_{j}=0$, for $j\geq 3$. The value of the constant $f_{NL}\times E\left(
T_{G}^{2}\right) $ is expected to be very small, namely of the order $%
10^{-4} $ \cite{Bart}. To simplify the discussion, we now assume that $%
ET_{G}^{2}=1$. In this case, by using (\ref{tmoe})--(\ref{tmoe3}), one has
that%
\begin{eqnarray*}
\widetilde{a}_{lm} &=&a_{lm}+f_{NL}a_{lm}(2)\text{ , } \\
a_{lm}(2) &=&\int_{S^{2}}T^{2}\overline{Y}_{lm}dx=\int_{S^{2}}\sum_{\ell
_{1}\ell _{2}}\sum_{m_{1}m_{2}}a_{\ell _{1}m_{1}}a_{\ell _{2}m_{2}}Y_{\ell
_{1}m_{1}}Y_{\ell _{2}m_{2}}\overline{Y}_{lm}dx \\
&=&\sum_{\ell _{1}\ell _{2}}\sum_{m_{1}m_{2}}a_{\ell _{1}m_{1}}a_{\ell
_{2}m_{2}}\sqrt{\frac{(2\ell _{1}+1)(2\ell _{2}+1)}{(2l+1)4\pi }}C_{\ell
_{1}0\ell _{2}0}^{l0}C_{\ell _{1}m_{1}\ell _{2}m_{2}}^{lm}\text{ .}
\end{eqnarray*}%
It follows that%
\begin{equation*}
\widetilde{C}_{l}:=E|\widetilde{a}_{lm}|^{2}=C_{l}+2f_{NL}^{2}%
\sum_{l_{1}l_{2}}C_{l_{1}}C_{l_{2}}\frac{(2l_{1}+1)(2l_{2}+1)}{4\pi (2l+1)}%
\left( C_{l_{1}0l_{2}0}^{l0}\right) ^{2},
\end{equation*}%
so that%
\begin{eqnarray*}
Var(T) &=&\sum_{l}\frac{2l+1}{4\pi }\widetilde{C}_{l}=\sum_{l}\frac{2l+1}{%
4\pi }C_{l}+2f_{NL}^{2}\sum_{l_{1}l_{2}}C_{l_{1}}C_{l_{2}}\frac{%
(2l_{1}+1)(2l_{2}+1)}{(4\pi )^{2}}\sum_{l}\left(
C_{l_{1}0l_{2}0}^{l0}\right) ^{2} \\
&=&\sum_{l}\frac{2l+1}{4\pi }C_{l}+2f_{NL}^{2}\left\{ \sum_{l_{1}}C_{l_{1}}%
\frac{(2l_{1}+1)}{4\pi }\right\} ^{2}=Var(T_{G})+f_{NL}^{2}Var(H_{2}(T_{G}))%
\text{ ,}
\end{eqnarray*}%
as expected, due to the orthogonality properties of Hermite polynomials. For
the bispectrum, we obtain therefore%
\begin{equation*}
E\widetilde{a}_{l_{1}m_{1}}\widetilde{a}_{l_{2}m_{2}}\widetilde{a}%
_{l_{3}m_{3}}=E\left\{
(a_{l_{1}m_{1}}+f_{2}a_{l_{1}m_{1}}(2))(a_{l_{2}m_{2}}+f_{2}a_{l_{2}m_{2}}(2))(a_{l_{3}m_{3}}+f_{2}a_{l_{3}m_{3}}(2))\right\}
\end{equation*}%
\begin{eqnarray*}
&=&f_{2}Ea_{l_{1}m_{1}}(2)a_{l_{2}m_{2}}a_{l_{3}m_{3}}+f_{2}Ea_{l_{1}m_{1}}a_{l_{2}m_{2}}(2)a_{l_{3}m_{3}}
\\
&&+f_{2}Ea_{l_{1}m_{1}}a_{l_{2}m_{2}}a_{l_{3}m_{3}}(2)+f_{2}^{3}Ea_{l_{1}m_{1}}(2)a_{l_{2}m_{2}}(2)a_{l_{3}m_{3}}(2)
\\
&=&(-1)^{m_{3}}C_{l_{1}m_{1}l_{2}m_{2}}^{l_{3}-m_{3}}P_{l_{1}l_{2}l_{3}}%
\text{ ,}
\end{eqnarray*}%
where%
\begin{equation}
P_{l_{1}l_{2}l_{3}}=6f_{2}\sqrt{\frac{(2l_{1}+1)(2l_{2}+1)}{(2l_{3}+1)4\pi }}%
C_{l_{1}0l_{2}0}^{l_{3}0}\left\{
C_{l_{1}}C_{l_{2}}+C_{l_{1}}C_{l_{3}}+C_{l_{2}}C_{l_{3}}\right\}
\label{bard1}
\end{equation}%
\begin{eqnarray}
&&+f_{2}^{3}\sum_{\ell _{1}\ell _{2}\ell _{3}}C_{\ell _{1}0\ell
_{2}0}^{l_{1}0}C_{\ell _{1}0\ell _{3}0}^{l_{2}0}C_{\ell _{2}0\ell
_{3}0}^{l_{3}0}\frac{(2\ell _{1}+1)(2\ell _{2}+1)(2\ell _{3}+1)}{\sqrt{(4\pi
)^{3}}}\times  \label{bard2} \\
&&\text{ \ \ \ \ \ \ \ \ \ \ \ \ \ \ \ \ \ \ \ \ \ \ \ \ \ \ \ \ \ \ \ \ \ \
\ \ \ \ \ \ \ \ \ \ }\times \frac{8(-1)^{l_{3}}}{\sqrt{2l_{3}+1}}\left\{
\begin{tabular}{lll}
$\ell _{1}$ & $\ell _{2}$ & $\ell _{3}$ \\
$l_{3}$ & $l_{2}$ & $l_{1}$%
\end{tabular}%
\right\} \left\{ C_{\ell _{1}}C_{\ell _{2}}C_{\ell _{3}}\right\} \text{ .}
\notag
\end{eqnarray}%
The lack of symmetry with respect to the $l_{3}$ term is only apparent and
can be easily dispensed with by permuting the multipoles in $%
C_{l_{1}m_{1}l_{2}m_{2}}^{l_{3}m_{3}}$ or using expression (\ref{clewig2}).
Formula (\ref{bard1}) is consistent with the cosmological literature, where (%
\ref{bard2}) is considered a higher order term and hence neglected (see
again (\cite{Hu})).

\subsection{The Connection with Higher Order Moments}

We now provide a simple result, connecting the reduced polyspectrum with the
higher order moments of the associated spherical random field.

\begin{proposition}
The following identity holds for every isotropic field with finite moments
of order $p$ and with a reduced polyspectrum $\left\{
P_{l_{1}...l_{p}}\left( \cdot \right) :l_{1},...,l_{p}\geq 0\right\} $: for
every $x\in S^{2}$,%
\begin{equation*}
ET\left( x\right) ^{p}\equiv \sum_{l_{1}...l_{p}}\sqrt{\frac{(2l_{1}+1)\cdot
\cdot \cdot (2l_{p}+1)}{(4\pi )^{p}}}\sum_{\lambda _{1}...\lambda
_{p-3}}P_{l_{1}...l_{p}}(\lambda _{1},...,\lambda
_{p-3})C_{l_{1}0...l_{p-2}0}^{\lambda _{1}...\lambda _{p-3}l_{p}0}\text{.}
\end{equation*}
\end{proposition}

\begin{proof}
We use the trivial fact that
\begin{equation*}
T(x)\overset{d}{=}T(0)=\sum_{l}a_{l0}Y_{l0}(0)=\sum_{l}a_{l0}\sqrt{\frac{2l+1%
}{4\pi }}\text{,}
\end{equation*}%
where $0$ is the North Pole and we used the fact that, for $m\neq 0$, $%
Y_{lm}\left( 0\right) =0$ and $Y_{l0}\left( 0\right) =\sqrt{\frac{2l+1}{4\pi
}}$ (see e.g. \cite[Chapter 5]{VMK}). Hence,
\begin{eqnarray*}
ET^{p} &=&\sum_{l_{1}...l_{p}}\sqrt{\frac{(2l_{1}+1)\cdot \cdot \cdot
(2l_{p}+1)}{(4\pi )^{p}}}E\left\{ a_{l_{1}0}...a_{l_{p}0}\right\} \\
&=&\sum_{l_{1}...l_{p}}\sqrt{\frac{(2l_{1}+1)\cdot \cdot \cdot (2l_{p}+1)}{%
(4\pi )^{p}}}\sum_{\lambda _{1}...\lambda _{p-3}}P_{l_{1}...l_{p}}(\lambda
_{1},...,\lambda _{p-3})C_{l_{1}0...l_{p-2}0}^{\lambda _{1}...\lambda
_{p-3}l_{p}0}.
\end{eqnarray*}
\end{proof}

\bigskip

\textbf{Example.} Take $T=H_{q}(T_{G})$, where $H_{q}$ is the $q$th Hermite
polynomial$.$ Then $ET^{p}=c_{pq}\left\{ ET^{2}\right\} ^{qp/2}$, where $%
c_{pq}\in \mathbb{N}$ denotes the number of Gaussian diagrams without flat
edges with $p$ rows and $q$ columns (see \cite{Surg}). Therefore, one has
the identity%
\begin{eqnarray*}
&&\sum_{l_{1}...l_{p}}\sqrt{\frac{(2l_{1}+1)...(2l_{p}+1)}{(4\pi )^{p}}}%
\sum_{\lambda _{1}...\lambda _{p-3}}P_{l_{1}...l_{p}}(\lambda
_{1},...,\lambda _{p-3})C_{l_{1}0...l_{p-2}0}^{\lambda _{1}...\lambda
_{p-3}l_{p}0} \\
&=&c_{pq}\left\{ \sum_{l}\frac{(2l+1)}{4\pi }C_{l}\right\} ^{pq/2}.
\end{eqnarray*}

\subsection{The $\protect\chi _{\protect\nu }^{2}$ polyspectrum}

Previously in (\ref{bard2}), we have implicitly derived the ``$\chi _{1}^{2}$
bispectrum'', that is, the bispectrum associated with a field of the type $%
T=H_{2}\left( T_{G}\right) $, where $T_{G}$ is Gaussian, centered, isotropic
and with unit variance. More precisely, with the notation (\ref{tmoe})--(\ref%
{tmoe3}), one deduces from (\ref{bard2}) that%
\begin{eqnarray}
&&Ea_{l_{1}m_{1}}(2)a_{l_{2}m_{2}}(2)a_{l_{3}m_{3}}(2) \\
&&=\sum_{\substack{ \ell _{1}\ell _{2}\ell _{3} \times \notag  \\
\ell _{4}\ell _{5}\ell _{6}}}\sum_{\mu _{1}...\mu _{6}}C_{\ell
_{1}0\ell _{2}0}^{l_{1}0}C_{\ell _{1}\mu _{1}\ell _{2}\mu
_{2}}^{l_{1}m_{1}}C_{\ell _{3}0\ell _{4}0}^{l_{2}0}C_{\ell _{3}\mu
_{3}\ell _{4}\mu _{4}}^{l_{2}m_{2}}C_{\ell _{5}0\ell
_{6}0}^{l_{3}0}C_{\ell _{5}\mu _{5}\ell _{6}\mu _{6}}^{l_{3}m_{3}}
\times \notag
\\
&&\times \sqrt{\frac{(2\ell _{1}+1)(2\ell _{2}+1)}{(2l_{1}+1)4\pi
}\frac{(2\ell
_{3}+1)(2\ell _{4}+1)}{(2l_{2}+1)4\pi }\frac{(2\ell _{5}+1)(2\ell _{6}+1)}{%
(2l_{3}+1)4\pi }} \times \notag \\&& \times E\left\{ a_{\ell _{1}\mu
_{1}}a_{\ell _{2}\mu _{2}}a_{\ell _{3}\mu _{3}}a_{\ell _{4}\mu
_{4}}a_{\ell _{5}\mu
_{5}}a_{\ell _{6}\mu _{6}}\right\}\notag \\
&& =8(-1)^{l_{3}-m_{3}}\sum_{\ell _{1}\ell _{2}\ell _{3}}C_{\ell
_{1}0\ell _{2}0}^{l_{1}0}C_{\ell _{1}0\ell _{3}0}^{l_{2}0}C_{\ell
_{2}0\ell _{3}0}^{l_{3}0} \frac{(2\ell _{1}+1)(2\ell _{2}+1)(2\ell
_{3}+1)}{\sqrt{(4\pi )^{3}}}\times \notag \\
&& \times
\frac{C_{l_{1}m_{1}l_{2}m_{2}}^{l_{3}-m_{3}}}{\sqrt{2l_{3}+1}}\left\{
\begin{tabular}{lll}
$\ell _{1}$ & $\ell _{2}$ & $\ell _{3}$ \\
$l_{3}$ & $l_{2}$ & $l_{1}$%
\end{tabular}%
\right\} \left\{ C_{\ell _{1}}C_{\ell _{2}}C_{\ell _{3}}\right\} \text{ ,}
\label{earlier}
\end{eqnarray}%
see \cite[p. 260 ; p. 454]{VMK}.\emph{\ }We now wish to extend these results
to polyspectra of order $p=4,5,6$ for random fields of the type $T=$ $%
H_{2}(T_{G})$, where (as above) $T_{G}$ is Gaussian, centered, isotropic and
with unit variance $.$ As anticipated, here we focus on cumulants instead of
moments. We have the following result.

\begin{proposition}
\label{h2c}The cumulant $\chi \left( a_{l_{1}m_{1}}\left( 2\right)
,...,a_{l_{p}m_{p}}\left( 2\right) \right) $ ($p=4,5,6$) associated with the
harmonic coefficients of an isotropic random field of the type $H_{2}\left(
T_{G}\right) $ (where $T_{G}$ is Gaussian and isotropic, with angular power
spectrum $\left\{ C_{l}:l\geq 0\right\} $) given by%
\begin{equation*}
\chi \left( a_{l_{1}m_{1}}\left( 2\right) ,...,a_{l_{p}m_{p}}\left( 2\right)
\right) =(-1)^{l_{p}-m_{p}}\sum_{\lambda _{1}...\lambda
_{p-3}}C_{l_{1}m_{1}...l_{p-1}m_{p-1}}^{\lambda _{1}...\lambda
_{p-3}l_{p}-m_{p}}\times P_{l_{1}...l_{p}}^{C;1}\left( \lambda
_{1},...,\lambda _{p-3}\right) ,
\end{equation*}%
where the reduced cumulant polyspectrum $\left\{ P_{l_{1}...l_{p}}^{C}\left(
\cdot \right) :l_{1},...,l_{p}\geq 0\right\} $ is given by%
\begin{equation*}
P_{l_{1}l_{2}l_{3}l_{4}}^{C;1}(\lambda )=48\sqrt{\frac{\left( 2\lambda
+1\right) }{(4\pi )^{4}(2l_{4}+1)}}\sum_{\ell _{1}...\ell _{4}}C_{\ell
_{1}}...C_{\ell _{4}}C_{\ell _{1}0\ell _{2}0}^{l_{1}0}C_{\ell _{2}0\ell
_{3}0}^{l_{3}0}C_{\ell _{3}0\ell _{4}0}^{l_{4}0}C_{\ell _{4}0\ell
_{1}0}^{l_{2}0}
\end{equation*}%
\begin{equation*}
\times (2\ell _{1}+1)...(2\ell _{4}+1)(-1)^{l_{1}+l_{2}+\ell _{2}+\ell
_{4}}\left\{
\begin{tabular}{lll}
$l_{1}$ & $l_{2}$ & $\lambda $ \\
$\ell _{4}$ & $\ell _{2}$ & $\ell _{1}$%
\end{tabular}%
\right\} \left\{
\begin{tabular}{lll}
$\lambda $ & $l_{3}$ & $l_{4}$ \\
$\ell _{3}$ & $\ell _{4}$ & $\ell _{2}$%
\end{tabular}%
\right\} \text{ for }p=4\text{ ,}
\end{equation*}%
\begin{equation*}
P_{l_{1}...l_{5}}^{C;1}(\lambda _{1},\lambda _{2})=384\sqrt{\frac{\left(
2\lambda _{1}+1\right) \left( 2\lambda _{2}+1\right) }{(4\pi )^{5}(2l_{5}+1)}%
}\sum_{\ell _{1}...\ell _{5}}C_{\ell _{1}}...C_{\ell _{5}}C_{\ell _{1}0\ell
_{2}0}^{l_{1}0}C_{\ell _{2}0\ell _{3}0}^{l_{2}0}C_{\ell _{3}0\ell
_{4}0}^{l_{4}0}C_{\ell _{4}0\ell _{5}0}^{l_{5}0}C_{\ell _{5}0\ell
_{1}0}^{l_{3}0}\times
\end{equation*}%
\begin{equation*}
\times (2\ell _{1}+1)...(2\ell _{5}+1)(-1)^{\ell _{1}+\ell
_{5}+l_{3}}\left\{
\begin{tabular}{lll}
$l_{1}$ & $l_{2}$ & $\lambda _{1}$ \\
$\ell _{3}$ & $\ell _{1}$ & $\ell _{2}$%
\end{tabular}%
\right\} \left\{
\begin{tabular}{lll}
$\lambda _{1}$ & $l_{3}$ & $\lambda _{2}$ \\
$\ell _{5}$ & $\ell _{3}$ & $\ell _{1}$%
\end{tabular}%
\right\} \left\{
\begin{tabular}{lll}
$\lambda _{2}$ & $l_{4}$ & $l_{5}$ \\
$\ell _{4}$ & $\ell _{5}$ & $\ell _{3}$%
\end{tabular}%
\right\} \text{ ,for }p=5\text{ ,}
\end{equation*}%
and%
\begin{equation*}
P_{l_{1}...l_{6}}^{C;1}(\lambda _{1},\lambda _{2},\lambda _{3})=3840\sqrt{%
\frac{\left( 2\lambda _{1}+1\right) \left( 2\lambda _{2}+1\right) \left(
2\lambda _{3}+1\right) }{(4\pi )^{6}(2l_{5}+1)}}\sum_{\ell _{1}...\ell
_{5}}C_{\ell _{1}}...C_{\ell _{6}}C_{\ell _{1}0\ell _{2}0}^{l_{1}0}C_{\ell
_{2}0\ell _{3}0}^{l_{2}0}C_{\ell _{3}0\ell _{4}0}^{l_{3}0}C_{\ell _{4}0\ell
_{5}0}^{l_{5}0}C_{\ell _{5}0\ell _{6}0}^{l_{6}0}C_{\ell _{6}0\ell
_{1}0}^{l_{4}0}\times
\end{equation*}%
\begin{equation*}
\times (2\ell _{1}+1)...(2\ell _{6}+1)(-1)^{\lambda _{1}+\ell _{3}+\ell
_{6}+l_{4}}\left\{
\begin{tabular}{lll}
$l_{1}$ & $l_{2}$ & $\lambda _{1}$ \\
$\ell _{3}$ & $\ell _{1}$ & $\ell _{2}$%
\end{tabular}%
\right\} \left\{
\begin{tabular}{lll}
$\lambda _{1}$ & $l_{5}$ & $\lambda _{2}$ \\
$\ell _{5}$ & $\ell _{3}$ & $\ell _{1}$%
\end{tabular}%
\right\} \left\{
\begin{tabular}{lll}
$\lambda _{2}$ & $l_{3}$ & $l_{4}$ \\
$\ell _{4}$ & $\ell _{5}$ & $\ell _{3}$%
\end{tabular}%
\right\} \text{ ,for }p=6\text{ .}
\end{equation*}
\end{proposition}

\begin{proof}
The result can be proved by means of the standard graphical techniques for
convolutions of Clebsch-Gordan coefficients, as described in \cite[Chapters
11 and 12]{VMK}. Here, we only provide the complete proof for the case $p=6$%
. Let $\left\{ a_{\ell m}\right\} $ be the random harmonic coefficients
associated with the underlying Gaussian field $T_{G}$. By definition, the
field $H_{2}\left( T_{G}\right) $ admits the expansion%
\begin{equation*}
H_{2}\left( T_{G}\right) =\sum_{l\geq 0}\sum_{m=-l}^{l}a_{lm}\left( 2\right)
Y_{lm},
\end{equation*}%
where
\begin{eqnarray*}
a_{lm}\left( 2\right) &=&\sum_{\ell _{1}m_{1}\ell _{2}m_{2}}a_{\ell
_{1}m_{1}}a_{\ell _{2}m_{2}}\int_{S^{2}}Y_{\ell _{1}m_{1}}\left( x\right)
Y_{\ell _{2}m_{2}}\left( x\right) \overline{Y_{lm}\left( x\right) }dx \\
&=&\sum_{\ell _{1}m_{1}\ell _{2}m_{2}}a_{\ell _{1}m_{1}}a_{\ell
_{2}m_{2}}\left(
\begin{array}{ccc}
\ell _{1} & \ell _{2} & l \\
m_{1} & m_{2} & -m%
\end{array}%
\right) \times \left( -1\right) ^{m}\times \\
&&\text{ \ \ \ \ \ \ \ \ \ \ \ \ \ \ \ \ \ \ \ \ \ \ \ \ }\times \left(
\begin{array}{ccc}
\ell _{1} & \ell _{2} & l \\
0 & 0 & 0%
\end{array}%
\right) \sqrt{\frac{\left( 2\ell _{1}+1\right) \left( 2\ell _{2}+1\right)
\left( 2l+1\right) }{4\pi }} \\
&=&\sum_{\ell _{1}m_{1}\ell _{2}m_{2}}a_{\ell _{1}m_{1}}a_{\ell
_{2}m_{2}}C_{\ell _{1}m_{1}\ell _{2}m_{2}}^{lm}C_{\ell _{1}0\ell _{2}0}^{lm}%
\sqrt{\frac{\left( 2\ell _{1}+1\right) \left( 2\ell _{2}+1\right) }{4\pi
\left( 2l+1\right) }}.
\end{eqnarray*}

By using once again the multilinearity of cumulants, one obtains that%
\begin{eqnarray*}
&&\mathrm{Cum}\left\{ a_{l_{1}m_{1}}\left( 2\right)
,...,a_{l_{6}m_{6}}\left( 2\right) \right\} \\
&=&\sum_{\ell _{11}m_{11}\ell _{12}m_{12}}\cdot \cdot \cdot \sum_{\ell
_{61}m_{61}\ell _{61}m_{61}}\mathrm{Cum}\left\{ a_{\ell _{11}m_{11}}a_{\ell
_{12}m_{12}},...,a_{\ell _{61}m_{61}}a_{\ell _{62}m_{62}}\right\} \times \\
&&\times \prod_{j=1}^{6}\left\{ C_{\ell _{j1}m_{j1}\ell
_{j2}m_{j2}}^{l_{j}m_{j}}C_{\ell _{j1}0\ell _{j2}0}^{l_{j}m_{j}}\sqrt{\frac{%
\left( 2\ell _{j1}+1\right) \left( 2\ell _{j2}+1\right) }{4\pi \left(
2l_{j}+1\right) }}\right\} .
\end{eqnarray*}%
For a given $\mathbf{lm}=\left( \ell _{11}m_{11},\ell _{12}m_{12};...;\ell
_{61}m_{61},\ell _{62}m_{62}\right) $, the quantity $\mathrm{Cum}\left\{
a_{\ell _{11}m_{11}}a_{\ell _{12}m_{12}},...,a_{\ell _{61}m_{61}}a_{\ell
_{62}m_{62}}\right\} $ is computed as follows:

\begin{itemize}
\item Build the $6\times 2$ matrix%
\begin{equation*}
\Lambda \left( \mathbf{lm}\right) =\left[
\begin{array}{cc}
\ell _{11}m_{11} & \ell _{12}m_{12} \\
\ell _{21}m_{21} & \ell _{22}m_{22} \\
\ell _{31}m_{31} & \ell _{32}m_{32} \\
\ell _{41}m_{41} & \ell _{42}m_{42} \\
\ell _{51}m_{51} & \ell _{52}m_{52} \\
\ell _{61}m_{61} & \ell _{62}m_{62}%
\end{array}%
\right]
\end{equation*}

\item Define the class $M\left( \Lambda \left( \mathbf{lm}\right) \right) $
of connected, Gaussian non-flat diagrams over $\Lambda $, that is, every $%
\gamma \in M\left( \Lambda \left( \mathbf{lm}\right) \right) $ is a
partition of the entries of $\Lambda \left( \mathbf{lm}\right) $, into pairs
belonging to different rows; moreover, such a partition has to be \textsl{%
connected}, in the sense that $\gamma $ cannot be divided into two separate
diagrams. For instance, an element of $M\left( \Lambda \left( \mathbf{lm}%
\right) \right) $ is
\begin{eqnarray*}
\gamma &=&\{\left\{ \ell _{11}m_{11},\ell _{21}m_{21}\right\} \left\{ \ell
_{22}m_{22},\ell _{32}m_{32}\right\} \left\{ \ell _{31}m_{31},\ell
_{41}m_{41}\right\} \\
&&\text{ \ \ \ \ \ \ \ \ \ \ \ \ \ \ \ \ \ \ \ \ \ \ \ \ \ \ \ \ \ \ }%
\left\{ \ell _{42}m_{42},\ell _{52}m_{52}\right\} \left\{ \ell
_{51}m_{61},\ell _{61}m_{61}\right\} \left\{ \ell _{62}m_{62},\ell
_{12}m_{12}\right\} \}
\end{eqnarray*}

\item For every $\gamma \in M\left( \Lambda \left( \mathbf{lm}\right)
\right) $, write
\begin{equation*}
\delta \left( \gamma \right) =\prod_{\left\{ \ell _{ab}m_{ab},\ell
_{cd}m_{cd}\right\} \in \gamma }\delta _{\ell _{cd}}^{\ell _{ab}}\delta
_{m_{ab}}^{-m_{cd}}\left( -1\right) ^{m_{ab}}C_{\ell _{ab}}
\end{equation*}%
(where $\delta _{a}^{b}$ is the usual Kronecker symbol)

\item Use the standard diagram formula (see again \cite{Surg}), to obtain
that%
\begin{equation*}
\mathrm{Cum}\left\{ a_{\ell _{11}m_{11}}a_{\ell _{12}m_{12}},...,a_{\ell
_{61}m_{61}}a_{\ell _{62}m_{62}}\right\} =\sum_{\gamma \in M\left( \Lambda
\left( \mathbf{lm}\right) \right) }\delta \left( \gamma \right) .
\end{equation*}
\end{itemize}

It follows that
\begin{eqnarray*}
&&\mathrm{Cum}\left\{ a_{l_{1}m_{1}}\left( 2\right)
,...,a_{l_{6}m_{6}}\left( 2\right) \right\}  \\
&=&\sum_{\mathbf{lm}}\sum_{\gamma \in M\left( \Lambda \left( \mathbf{lm}%
\right) \right) }\delta \left( \gamma \right) \prod_{j=1}^{6}\left\{ C_{\ell
_{j1}m_{j1}\ell _{j2}m_{j2}}^{l_{j}m_{j}}C_{\ell _{j1}0\ell
_{j2}0}^{l_{j}m_{j}}\sqrt{\frac{\left( 2\ell _{j1}+1\right) \left( 2\ell
_{j2}+1\right) }{4\pi \left( 2l_{j}+1\right) }}\right\} ,
\end{eqnarray*}%
where the first sum runs over all vectors of the type $\mathbf{lm}=\left(
\ell _{11}m_{11},\ell _{12}m_{12};...;\ell _{61}m_{61},\ell
_{62}m_{62}\right) $. The proof now follows directly from graphical
techniques. In particular, the previous term can be associated with an
hexagon, having in each vertex an outward line corresponding to a ``free''\
(i.e. not summed up) index $l_{i}m_{i}$, $i=1,...,6$. An expression for
convolutions of Clebsch-Gordan coefficients corresponding to such a
configuration can be found in \cite[p. 461]{VMK}, eq. 12.1.6.30. From this,
standard combinatorial arguments and a convenient relabelling of the
indexes, we obtain that
\begin{eqnarray*}
P_{l_{1}...l_{6}}^{C;1}(\lambda _{1},\lambda _{2},\lambda _{3}) &=&3840\sqrt{%
\frac{\left\{ \prod_{j=1}^{3}\left( 2\lambda _{j}+1\right) \right\} }{(4\pi
)^{6}(2l_{p}+1)}}\times (-1)^{\lambda _{1}+\ell _{3}+\ell _{6}+l_{4}} \\
&&\times \sum_{\ell _{1}...\ell _{6}}(2\ell _{1}+1)\cdot \cdot \cdot (2\ell
_{6}+1)C_{\ell _{1}}...C_{\ell _{6}}C_{\ell _{1}0\ell _{2}0}^{l_{1}0}C_{\ell
_{2}0\ell _{3}0}^{l_{2}0}C_{\ell _{3}0\ell _{4}0}^{l_{3}0}C_{\ell _{4}0\ell
_{5}0}^{l_{5}0}C_{\ell _{5}0\ell _{6}0}^{l_{6}0}C_{\ell _{6}0\ell
_{1}0}^{l_{4}0} \\
&&\times \left\{
\begin{tabular}{lll}
$l_{1}$ & $l_{2}$ & $\lambda _{1}$ \\
$\ell _{3}$ & $\ell _{1}$ & $\ell _{2}$%
\end{tabular}%
\right\} \left\{
\begin{tabular}{lll}
$\lambda _{1}$ & $\lambda _{2}$ & $l_{3}$ \\
$\ell _{4}$ & $\ell _{3}$ & $\ell _{1}$%
\end{tabular}%
\right\} \left\{
\begin{tabular}{lll}
$\lambda _{2}$ & $l_{4}$ & $\lambda _{3}$ \\
$\ell _{6}$ & $\ell _{4}$ & $\ell _{1}$%
\end{tabular}%
\right\} \left\{
\begin{tabular}{lll}
$\lambda _{3}$ & $l_{5}$ & $l_{6}$ \\
$\ell _{5}$ & $\ell _{6}$ & $\ell _{4}$%
\end{tabular}%
\right\} .
\end{eqnarray*}%
Note that $3840=2^{p-1}\left( p-1\right) !=2^{5}5!$ is the number of
automorphisms between graphs belonging to $M\left( \Lambda \left( \mathbf{lm}%
\right) \right) $.
\end{proof}

\bigskip

We recall that the Clebsch-Gordan coefficients $\left\{
C_{a0b0}^{c0}\right\} $ are identically zero unless $a+b+c$ is even; it is
hence easy to see that the previous polyspectra are non-zero only if the sum
$\left\{ l_{1}+...+l_{p}\right\} $ is even as well.

From the previous Proposition, we can derive the corresponding expressions
for the cumulant polyspectra for $\chi _{\nu }^{2}$ random field.

\bigskip

\textbf{Definition B. }We say the random field $T_{\chi _{\nu }^{2}}$ has a
chi-square law with $\nu \geq 1$ degrees of freedom if there exist $\nu $
independent and identically distributed Gaussian random fields $T_{i}$ such
that%
\begin{equation*}
T_{\chi _{\nu }^{2}}\overset{law}{=}T_{1}^{2}+...+T_{\nu }^{2}\text{ .}
\end{equation*}

\bigskip

It is trivial to show that $T_{\chi _{\nu }^{2}}$ is mean-square continuous
and isotropic if $T_{i}$ is. We have the following

\begin{proposition}
The cumulant polyspectra of $T_{\chi _{\nu }^{2}}$ (for $p\geq 2$) are given
by%
\begin{equation*}
P_{l_{1}...l_{p}}^{C;\nu }(\lambda _{1},...,\lambda _{p-3})=\nu
P_{l_{1}...l_{p}}^{C;1}(\lambda _{1},...,\lambda _{p-3}).
\end{equation*}
\end{proposition}
\begin{proof}
Note that the cumulant polyspectra of order $p\geq 2$ of $T_{\chi _{\nu
}^{2}}$ coincide with those of the centered field $T_{\chi _{\nu
}^{2}}-ET_{\chi _{\nu }^{2}}$ (due to the translation-invariance properties
of cumulants). Then, the proof is an immediate consequence of Proposition %
\ref{h2c} and the of the standard multinearity properties of cumulants.
\end{proof}

\section{Further Issues and Applications \label{further}}

The purpose of this final Section is to introduce what we view as promising
directions for further research, where the ideas of this paper may perhaps
yield further insights. We shall delay to future work a more thorough
investigation of the issues which are left open below.

\subsection{Representations of the Symmetric Group\label{SS : reprSym}}

As a further link between representation theory and higher order angular
power spectra, we mention the following. It is to be stressed that the
decomposition of $\Delta _{l_{1}...l_{n}}$ that we achieved in the previous
Proposition \ref{propb} is by no means unique. In particular, what we did
was to choose a particular sequence of \ \textquotedblleft
couplings\textquotedblright , i.e. we partitioned tensor products of the
Wigner's matrices $D^{l}$ in a specific order before decomposing them into
direct sums. Alternative partitions yield different eigenvectors and
therefore, different expressions for the polyspectra/joint moments .
Alternatively, we could maintain the same coupling scheme (for instance,
\textquotedblleft start always from the first pair on the
left\textquotedblright , as we did earlier) but acting on $(l_{1},...,l_{n})$
by the symmetric group $S_{n}.$ However, not all coupling schemes can be
achieved by simply permuting the elements of $(l_{1},l_{2},...,l_{n}).$ This
is the well-known \emph{problem of parentheses }in Mathematical Physics (see
for instance \cite{BieLou}).

We suggest here that one can establish a link between alternate expressions
for the angular polyspectra and representations of the symmetric group. More
precisely the alternate expressions that we find for the polyspectra $%
P_{l_{1}....l_{n}}(\lambda _{1},...,\lambda _{n-3})$ of a strongly isotropic
field (with $n$-moments) must be such that, for every permutation $\pi \in
\mathfrak{S}_{n}$,%
\begin{eqnarray*}
&&\sum_{\lambda _{1}}...\sum_{\lambda
_{n-3}}C_{l_{1}m_{1}....l_{n-1}m_{n-1}}^{\lambda _{1}...\lambda
_{n-3};l_{n}-m_{n}}P_{l_{1}....l_{n}}(\lambda _{1},...,\lambda _{n-3}) \\
&=&\sum_{\lambda _{1}^{\prime }}...\sum_{\lambda _{n-3}^{\prime }}C_{\pi
(l_{1})m_{1}....\pi (l_{n-1})m_{n-1}}^{\lambda _{1}^{\prime }...\lambda
_{n-3}^{\prime };l_{n}-m_{n}}P_{\pi (l_{1})....\pi (l_{n})}(\lambda
_{1}^{\prime },...,\lambda _{n-3}^{\prime })\text{ .}
\end{eqnarray*}%
Now let us multiply both sides by $C_{l_{1}m_{1}....l_{n-1}m_{n-1}}^{\lambda
_{1}^{\prime \prime }...\lambda _{n-3}^{\prime \prime };l_{n}m_{n}^{\prime
}} $, where $(\lambda _{1}^{\prime \prime },...,\lambda _{n-3}^{\prime
\prime }) $ is fixed, and sum over $(m_{1},...m_{n}).$ In view of the
unitary properties of Clebsch-Gordan coefficients we obtain for the
left-hand side%
\begin{eqnarray}
&&\sum_{m_{1}...m_{n}}C_{l_{1}m_{1}....l_{n-1}m_{n-1}}^{\lambda _{1}^{\prime
\prime }...\lambda _{n-3}^{\prime \prime };l_{n}-m_{n}}\left\{ \sum_{\lambda
_{1}}...\sum_{\lambda _{n-3}}C_{l_{1}m_{1}....l_{n-1}m_{n-1}}^{\lambda
_{1}...\lambda _{n-3};l_{n}-m_{n}}P_{l_{1}....l_{n}}(\lambda
_{1},...,\lambda _{n-3})\right\}  \notag \\
&=&\sum_{\lambda _{1}}...\sum_{\lambda _{n-3}}\left\{
\sum_{m_{1}...m_{n}}C_{l_{1}m_{1}....l_{n-1}m_{n-1}}^{\lambda _{1}^{\prime
\prime }...\lambda _{n-3}^{\prime \prime
};l_{n}-m_{n}}C_{l_{1}m_{1}....l_{n-1}m_{n-1}}^{\lambda _{1}...\lambda
_{n-3};l_{n}-m_{n}}P_{l_{1}....l_{n}}(\lambda _{1},...,\lambda
_{n-3})\right\}  \notag \\
&=&\sum_{\lambda _{1}}...\sum_{\lambda _{n-3}}\left\{ \delta _{\lambda
_{1}}^{\lambda _{1}^{\prime \prime }}...\delta _{\lambda _{n-3}}^{\lambda
_{n-3}^{\prime \prime }}P_{l_{1}....l_{n}}(\lambda _{1},...,\lambda
_{n-3})\right\} =P_{l_{1}....l_{n}}(\lambda _{1}^{\prime \prime
},...,\lambda _{n-3}^{\prime \prime })\text{;}  \label{franco pippo}
\end{eqnarray}%
on the right-hand side we get%
\begin{eqnarray}
&&\sum_{m_{1}...m_{n}}C_{l_{1}m_{1}....l_{n-1}m_{n-1}}^{\lambda _{1}^{\prime
\prime }...\lambda _{n-3}^{\prime \prime };l_{n}m_{n}}\left\{ \sum_{\lambda
_{1}^{\prime }}...\sum_{\lambda _{n-3}^{\prime }}C_{\pi (l_{1})m_{1}....\pi
(l_{n-1})m_{n-1}}^{\lambda _{1}^{\prime }...\lambda _{n-3}^{\prime
};l_{n}m_{n}}P_{\pi (l_{1})....\pi (l_{n})}(\lambda _{1}^{\prime
},...,\lambda _{n-3}^{\prime })\right\}  \notag \\
&=&\!\!\!\sum_{\lambda _{1}^{\prime }}...\sum_{\lambda
_{n-3}^{\prime }}\sum_{m_{1}...m_{n}}C_{\pi \left( l_{1}\right)
m_{1}....\pi \left( l_{n-1}\right) m_{n-1}}^{\lambda _{1}^{\prime
\prime }...\lambda _{n-3}^{\prime \prime };l_{n}m_{n}}C_{\pi
(l_{1})m_{1}....\pi (l_{n-1})m_{n-1}}^{\lambda _{1}^{\prime
}...\lambda _{n-3}^{\prime };l_{n}m_{n}}P_{\pi (l_{1})....\pi
(l_{n})}(\lambda _{1}^{\prime },...,\lambda _{n-3}^{\prime
})\text{.}  \label{pippo franco}
\end{eqnarray}%
Similarly as in the previous section, the sum of products of Clebsch-Gordan
coefficients on the right hand side can be expressed in terms of higher
order Wigner's coefficients. Since this section is just informal, for
brevity's sake we do not give explicit expressions (see e.g. \cite[ Chapter
10]{VMK}). The two expressions (\ref{franco pippo}) and (\ref{pippo franco})
imply that, for every fixed $\left( l_{1},...,l_{n}\right) $ and every
permutation $\pi $, there exists a square matrix $A\left( \left(
l_{1},...,l_{n}\right) ;\pi \right) $ such that
\begin{equation*}
P_{l_{1}....l_{n}}=A\left\{ \left( l_{1},...,l_{n}\right) ;\pi \right\}
P_{\pi (l_{1})....\pi (l_{n})}\text{,}
\end{equation*}%
where $P_{l_{1}...l_{n}}$ is the vector with entries $P_{l_{1}...l_{n}}%
\left( \lambda _{1},...,\lambda _{n}\right) $. We conjecture that in this
way one can build a representation of the symmetric group $\mathfrak{S}_{n}$
on the vector space generated by admissible polyspectra $P_{l_{1}....l_{n}}.$
If this is indeed the case, some important questions are left open: for
instance, whether or not the representation is \emph{faithful }(see \cite%
{Diaconis}), and whether these ideas can lead to algorithms for the
numerical simulation of representation matrices, along the lines of what we
shall pursue in the next subsection.

\subsection{Random data compression}

In this subsection we shall show how we can exploit the previous results to
develop a probabilistic algorithm to compress information on Clebsch-Gordan
coefficients. Note first that
\begin{equation*}
\#\left\{ C_{l_{1}m_{1}l_{2}m_{2}}^{l_{3}m_{3}}:l_{1},l_{2},l_{3}\leq
L,\left\vert C_{l_{1}m_{1}l_{2}m_{2}}^{l_{3}m_{3}}\right\vert \neq 0\right\}
\approx O(L^{6})\text{ ;}
\end{equation*}%
it is therefore clear how for most applications the storage of
Clebsch-Gordan coefficients for future usage is simply unfeasible, whatever
the supercomputing facilities (for instance, for CMB data analysis, $%
L\approx 3\times 10^{3}$ is currently required, so that the number of
Clebsch-Gordan coefficients to be saved would exceed $10^{20}$). Let us
consider again a chi-square field as defined before, i.e.%
\begin{equation*}
T_{\chi ^{2}}(x)=H_{2}(T_{G}(x))=\sum_{lm}a_{lm}(2)Y_{lm}(x)\text{ ;}
\end{equation*}%
we have proved earlier in (\ref{earlier}) that%
\begin{equation*}
Ea_{l_{1}m_{1}}(2)a_{l_{2}m_{2}}(2)a_{l_{3}m_{3}}(2)=(-1)^{m_{3}}C_{l_{1}m_{1}l_{2}m_{2}}^{l_{3}m_{3}}h_{l_{1}l_{2}l_{3}}
\end{equation*}%
where
\begin{equation*}
h_{l_{1}l_{2}l_{3}}:=8\sum_{\ell _{1}\ell _{2}\ell _{3}}C_{\ell _{1}0\ell
_{2}0}^{l_{1}0}C_{\ell _{1}0\ell _{3}0}^{l_{2}0}C_{\ell _{2}0\ell
_{3}0}^{l_{3}0}\frac{(2\ell _{1}+1)(2\ell _{2}+1)(2\ell _{3}+1)}{\sqrt{(4\pi
)^{3}}}\frac{1}{\sqrt{2l_{3}+1}}\left\{
\begin{tabular}{lll}
$\ell _{1}$ & $\ell _{2}$ & $\ell _{3}$ \\
$l_{1}$ & $l_{2}$ & $l_{3}$%
\end{tabular}%
\right\} \left\{ C_{\ell _{1}}C_{\ell _{2}}C_{\ell _{3}}\right\} \text{ ,}
\end{equation*}%
which can be calculated analytically and stored, with storage dimension%
\begin{equation*}
\#\left\{ h_{l_{1}l_{2}l_{3}}:l_{1},l_{2},l_{3}\leq L,\left\vert
C_{l_{1}0l_{2}0}^{l_{3}0}\right\vert \neq 0\right\} \approx O(L^{3})\text{ .}
\end{equation*}%
Let us assume we simulate $B$ times $T_{\chi ^{2}}(x),$ which is trivially
done by simply squaring a Gaussian field: the latter is obtained by sampling
independent complex Gaussian variables with variance $C_{l}.$ We store the
triangular arrays $\left\{ a_{lm}^{i}\right\} _{l=1,...,L;m=-l,...,l},$ $%
i=1,...,B$; here the dimension is of order $B\times L^{2}.$ We can then
recover any value $C_{l_{1}m_{1}l_{2}m_{2}}^{l_{3}m_{3}}$ by means of the
Monte Carlo estimate%
\begin{equation*}
\widehat{C}_{l_{1}m_{1}l_{2}m_{2}}^{l_{3}m_{3}}=h_{l_{1}l_{2}l_{3}}^{-1}%
\sum_{i=1}^{B}\frac{%
a_{l_{1}m_{1}}^{(i)}a_{l_{2}m_{2}}^{(i)}a_{l_{3}m_{3}}^{(i)}}{B}\text{ ,}
\end{equation*}%
which requires $B$ steps and $B\times L^{2}+L^{3}$ storage capacity, as
opposed to $L^{6}$ storage capacity by the direct method. We leave for
further research a more thorough investigation on the convergence properties
of this algorithm; we stress, however, that the procedure we advocate is
completely general, i.e. it does not depend on peculiar features of the
group $SO(3)$ we are currently considering. We believe, hence, that similar
ideas can be implemented for the numerical estimation of Clebsch-Gordan
coefficients for other compact groups of interest for theoretical
physicists. We leave this and the previous issues in this Section as topics
for further research.

\end{document}